\documentclass{amsart}

\usepackage{graphicx}
\usepackage{amssymb,amsmath,a4wide}

\usepackage{xypic}
\xyoption{all}
\xyoption{arc}
\xyoption{poly}
\newenvironment{claim}[1]
    {\par\addvspace{\medskipamount}\noindent\emph{#1.}\em\ \ignorespaces}
    {\normalfont\par\addvspace{\medskipamount}}

\theoremstyle{plain}
\newtheorem{theorem}{Theorem}[section]
\newtheorem{lemma}[theorem]{Lemma}
\newtheorem{corollary}[theorem]{Corollary}
\newtheorem{proposition}[theorem]{Proposition}

\theoremstyle{definition}
\newtheorem{definition}[theorem]{Definition}
\newtheorem{remark}[theorem]{Remark}

\title{Braid groups of imprimitive complex reflection groups}

\author{Ruth Corran}
\author{Eon-Kyung Lee}
\author{Sang-Jin Lee}

\address{The American University of Paris, 147 rue de Grenelle, 75007, France}
\email{Ruth.Corran@AUP.fr}

\address{Department of Mathematics, Sejong University, Seoul, 143-747, Korea}
\email[corresponding author]{eonkyung@sejong.ac.kr}

\address{Department of Mathematics, Konkuk University, Seoul, 143-701, Korea}
\email{sangjin@konkuk.ac.kr}

\begin{document}

\begin{abstract}
We obtain new presentations for the imprimitive complex reflection groups
of type $(de,e,r)$ and their braid groups $B(de,e,r)$ for $d,r\ge 2$.
Diagrams for these presentations are proposed.
The presentations have much in common with Coxeter presentations of
real reflection groups.
They are positive and homogeneous, and give rise to quasi-Garside structures.
Diagram automorphisms correspond to group automorphisms.
The new presentation shows how the braid group $B(de,e,r)$
is a semidirect product of the braid group of affine type $\widetilde{\mathbf{A}}_{r-1}$
and an infinite cyclic group.
Elements of $B(de,e,r)$ are visualized
as geometric braids on $r+1$ strings
whose first string is pure and whose winding number is a multiple of $e$.
We classify periodic elements, and show that the roots are unique up to conjugacy
and that the braid group $B(de,e,r)$ is strongly translation discrete.

\end{abstract}

\maketitle

\tableofcontents

\section{Introduction}
\label{sec:Introd}

\subsection{Reflection groups and braid groups}

A complex reflection group $G$ on a finite dimensional complex vector space $V$
is a subgroup of $\mbox{GL}(V)$ generated by complex reflections---nontrivial
elements that fix a complex hyperplane in $V$ pointwise.
Finite (irreducible) complex reflection groups were classified
by Shephard and Todd~\cite{ST54}:
\begin{enumerate}
\item
a general infinite family $G(de,e,r)$ for positive integral parameters $d, e, r$;

\item
34 exceptions, labeled $G_4, G_5, \ldots, G_{37}$.
\end{enumerate}
For the presentations of the above groups, see~\cite{BMR98}.

Finite complex reflection groups are divided into two main classes:
primitive and imprimitive.
The general infinite family $G(de,e,r)$ are imprimitive except $G(1,1,r)$ and $G(de,e,1)$.
($G(1,1,r)$ is the symmetric group of degree $r$ and
$G(de,e,1)$ is the cyclic group of order $d$.)
The exceptional groups $G_4, G_5, \ldots, G_{37}$ are primitive.

The complex reflection group of type $(de,e,r)$ is defined as
$$G(de,e,r) = \left\{\, \begin{array}{cc}
\mbox{$r\times r$ monomial matrices}\\
\mbox{$(x_{ij})$ over $\{0\} \cup \mu_{de}$}
\end{array} \Biggm|
\prod_{x_{ij}\neq 0} x_{ij}^d = 1
\,\right\},$$
where $\mu_{de}$ is the set of $de$-th roots of unity.
Special cases of $G(de,e,r)$ are isomorphic to real reflection groups:
$G(1,1,r)\cong G({\mathbf{A}}_{r-1})$,
$G(2,1,r)\cong G({\mathbf{B}}_{r})$,
$G(2,2,r)\cong G(\mathbf{D}_{r})$ and
$G(e,e,2)\cong G(\mathbf{I}_2(e))$,
where $G(W)$ denotes the Coxeter group of type $W$.
For all the other parameters, $G(de,e,r)$ has no real structure.

The braid group of a complex reflection group is defined as the
fundamental group of the regular orbits.
For these braid groups, the presentations and the centers are shown
in~\cite{BMR98, BDM02, BM04, Bes06b}.

The braid groups $B(de,e,r)$ of the complex reflection groups $G(de,e,r)$
are divided into two cases: $d=1$ and $d\ge 2$.
For any $d, d'\ge 2$, $B(de,e,r)=B(d'e,e,r)\neq B(e,e,r)$~\cite{BMR98}.
It was shown in~\cite{BC06, CP11} that the braid groups $B(e,e,r)$
are Garside groups.

\subsection{Outline and main results}

In this paper, we propose new presentations and diagrams
for the imprimitive reflection groups
$G(de,e,r)$ and their braid groups $B(de,e,r)$ for $d,r\ge 2$ and $e\ge 1$.
These presentations have much in common with Coxeter presentations
of real reflection groups.
They are positive and homogeneous, and give rise to quasi-Garside structures.
Diagram automorphisms correspond to group automorphisms.

To motivate our approach, we review in \S\ref{sec:prelim} the presentations
for the free group $F_2$ and the braid groups $B(\mathbf{I}_2(e))$ and $B(e,e,r)$.

As a generic version of the imprimitive reflection groups $G(e,e,r)$,
Shi~\cite{Shi02} introduced the complex reflection group $G(\infty,\infty,r)$
and showed that $G(\infty,\infty,r)$ is isomorphic to the affine
reflection group of type ${\widetilde{\mathbf{A}}}_{r-1}$.
In \S\ref{sec:Binfty},
we propose  new presentations for $G(\infty,\infty,r)$
and its braid group $B(\infty,\infty,r)$.
We show how the braid group $B(de,e,r)$ is a semidirect product
of $B(\infty, \infty, r)$ and an infinite cyclic group,
and then show how $G(de,e,r)$ is a semidirect product of $G(de, de, r)$
and a cyclic group of order $d$.
The new presentations give rise to quasi-Garside structures on the braid groups
$B(\infty, \infty, r)$ and $B(de,e,r)$.

In \S\ref{sec:GeomBraid}, we explore some properties of $B(de,e,r)$.
Elements of $B(de,e,r)$
will be interpreted as geometric braids
on $r+1$ strings whose first string is pure and whose winding number
around the first string is a multiple of $e$.
Using this interpretation, we show that the $k$-th root of an element of $B(de,e,r)$,
if exists, is unique up to conjugacy for any nonzero integer $k$,
show that $B(de,e,r)$ is strongly translation discrete,
and classify periodic elements of $B(de,e,r)$.

The braid group $B(\infty,\infty,r)$ is isomorphic to $B({\widetilde{\mathbf{A}}}_{r-1})$.
In \S\ref{sec:Atilde}, we propose an ${\widetilde{\mathbf{A}}}$-type presentation
for $B(de,e,r)$ which is also positive and homogeneous.

\smallskip
\emph{Notations.} \
Denote by $\langle  b_1 b_2 \cdots b_l\rangle^k$  the cyclic  product
$b_1 b_2 \cdots b_l b_1 b_2 \cdots $ with  $k$ factors.
For example, $\langle  b_1 b_2 b_3\rangle^2 = b_1 b_2$ and
$\langle b_1 b_2 b_3\rangle^5 = b_1 b_2 b_3 b_1 b_2$.
If $w$ is the word $b_1 b_2 \ldots b_l$, then we will write
$\langle w\rangle^k$ as a shorthand for $\langle b_1 b_2 \cdots b_l\rangle^k$.

\section{Preliminary material}
\label{sec:prelim}

\subsection{The free group $F_2$ and the braid group of type $\mathbf{I}_2(e)$}
\label{sec:F2_I2e}

Here we review the presentations of
the free group $F_2$ and the braid group of type $\mathbf{I}_2(e)$.
These presentations are not necessary for the work of this paper,
but they will give an intuition about our approach to the
braid group $B(de,e,r)$.

\subsubsection{Presentations for the free group $F_2$}

Consider the following two presentations for the free group $F_2$:
\begin{align*}
F_2 &=\langle \, t_0,t_1\mid \quad\rangle;\\
F_2 & = \langle \, t_i,\ i\in{\mathbb Z}\mid
t_it_{i-1}=t_jt_{j-1}\ \mbox{ for all $i,j\in{\mathbb Z}$}\,\rangle.
\end{align*}
The second presentation is obtained from the first one
by adding new generators $t_i$ for $i\in{\mathbb Z}\setminus\{0,1\}$
together with defining relations $\cdots =t_2t_1=t_1t_0=t_0t_{-1}=\cdots$.
These presentations can be described as in Figure~\ref{fig:FreeGp}(a,b).
The first diagram is also known as ${\widetilde{\mathbf{A}}}_1$.
In the diagram for the second presentation,
countably many nodes are tangent to a line labeled 2,
which means the relation $t_it_{i-1}=t_jt_{j-1}$ for $i,j\in{\mathbb Z}$.
It is known that the second presentation gives
rise to a quasi-Garside structure on $F_2$~\cite{Bes06a,DDGKM14}.
One may regard it as a \emph{dual} presentation of the first one.

\begin{figure}
$$\begin{array}{ccc}
\begin{xy}
(0,0) *++={\rule{0pt}{4pt}} *\frm{o};
(20,0) *++={\rule{0pt}{4pt}} *\frm{o}**@{-}  ?(0.5) *!/_2.5mm/{\infty};
(0,-4) *++={t_0};
(20,-4) *++={t_1}
\end{xy}&&
\begin{xy}
(-30,2) *++={\cdots};
( 30,2) *++={\cdots} **@{-};
(-20,0) *++={\rule{0pt}{4pt}} *\frm{o};
(-10,0) *++={\rule{0pt}{4pt}} *\frm{o};
(0,0) *++={\rule{0pt}{4pt}} *\frm{o};
(10,0) *++={\rule{0pt}{4pt}} *\frm{o};
(20,0) *++={\rule{0pt}{4pt}} *\frm{o};
(10,0) *++={\rule{0pt}{4pt}} *\frm{o};
(25,4) *++={2};
(-20,-4) *++={t_{-2}};
(-10,-4) *++={t_{-1}};
(0,-4) *++={t_0};
(10,-4) *++={t_1};
(20,-4) *++={t_2};
\end{xy}\\
\mbox{(a) Diagram for $F_2$} &\qquad\qquad&
\mbox{(b) Dual diagram for $F_2$}\\[1em]
\begin{xy}
(0,0) *++={\rule{0pt}{4pt}} *\frm{o};
(20,0) *++={\rule{0pt}{4pt}} *\frm{o}**@{-}  ?(0.5) *!/_2.5mm/{e};
(0,-4) *++={t_0};
(20,-4) *++={t_1}
\end{xy}&&
\begin{xy}
(-2, 0) *{\ellipse(12.5,12.5){}};
(-14,0) *++={2};
(3.8, 12) *++={\rule{0pt}{4pt}} *\frm{o}; (8, 11) *++={t_1};
(9.5, 5) *++={\rule{0pt}{4pt}} *\frm{o}; (13, 3) *++={t_0};
(9.5,-5) *++={\rule{0pt}{4pt}} *\frm{o}; (14,-8) *++={t_{e-1}};
(4,-12) *++={\rule{0pt}{4pt}} *\frm{o}; (10,-14) *++={t_{e-2}};
(-4,15) *{\cdots};
(-4,-15) *{\cdots};
\end{xy}\\
\mbox{(c) Diagram for $B(\mathbf{I}_2(e))$} &\qquad\qquad&
\mbox{(d) Dual diagram for $B(\mathbf{I}_2(e))$}
\end{array}
$$
\caption{Diagrams for $F_2$ and $B(\mathbf{I}_2(e))$}
\label{fig:FreeGp}
\end{figure}
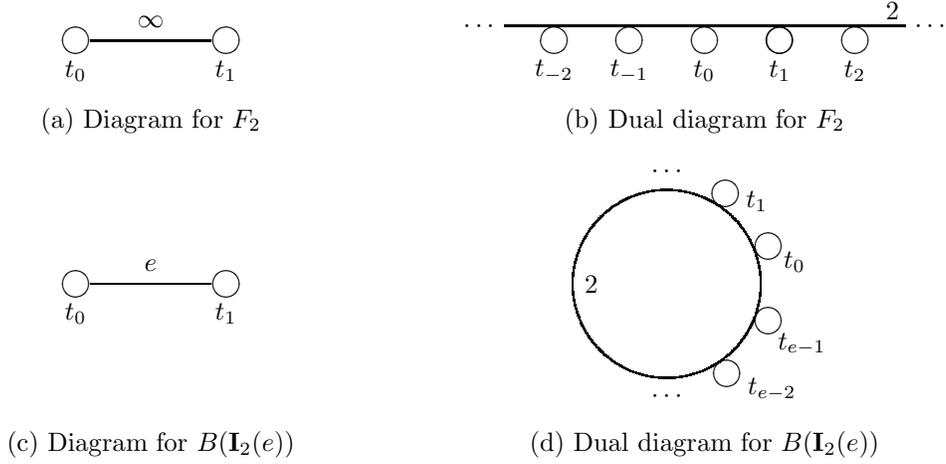

\subsubsection{Presentations for the braid group $B(\mathbf{I}_2(e))$}

The following are well-known presentations for the braid group of
the dihedral group on $2e$ elements, denoted $B(\mathbf{I}_2(e))$:
\begin{align*}
B(\mathbf{I}_2(e)) &= \langle \, t_0,t_1\mid
\langle  t_0t_1\rangle^e=\langle  t_1t_0\rangle^e\,\rangle;\\
B(\mathbf{I}_2(e)) &=
\langle \, t_0,t_1,\ldots,t_{e-1}\mid
t_1t_0=t_2t_1=\cdots=t_{e-1}t_{e-2}=t_0t_{e-1}\,\rangle.
\end{align*}
See Figure~\ref{fig:FreeGp}(c,d).
It is easy to see that the above two presentations are equivalent.
The first presentation is usually referred as the \emph{classical}
presentation, and the second as the \emph{dual} presentation.
Both presentations give rise to Garside structures.

The free group $F_2$ can be considered as a version of $B(\mathbf{I}_2(e))$
where $e$ is replaced by $\infty$.

\subsection{The braid group $B(e,e,r)$}
\label{sec:Beer}

Brou\'e, Malle and Rouquier~\cite{BMR98} obtained the following
presentation for the braid group  $B(e,e,r)$ for $e,r\ge 2$:
\begin{itemize}
\item Generators:  $\{t_0,t_1\} \cup S$ where $S = \{s_3,\ldots,s_r\}$;
\item Relations: the usual braid relations on $S$, along with\\
$\begin{array}{ll}
(P_1) & \langle t_1 t_0\rangle^e = \langle t_0 t_1\rangle^{e}, \\
(P_2) & s_3 t_i s_3 =t_i s_3 t_i \quad \mbox{for  $i = 0,1$}, \\
(P_3) & s_j t_i = t_i s_j \quad \mbox{for $i = 0,1$ and $4 \leq j \leq r$},\\
(P_4) & s_3(t_1 t_0)s_3(t_1 t_0) = (t_1 t_0)s_3(t_1 t_0)s_3.
\end{array}$
\end{itemize}

Furthermore, a presentation for the reflection group $G(e,e,r)$
is obtained by adding the relation
$a^2=1$ for all generators $a$, and the generators
are then all reflections.

The presentation is usually illustrated by the diagram shown
in Figure~\ref{fig:Beer}(a).
The diagram is to be read as a Coxeter graph for the real
reflection group case:
when nodes $a$ and $b$ are joined by an edge labeled $e$,
there is a relation $\langle ab\rangle^e=\langle ba\rangle^e$;
when nodes $a$ and $b$ are joined by an unlabelled edge,
there is a relation $aba=bab$;
when two nodes $a$ and $b$ are not connected by an edge,
there is a relation $ab=ba$;
the double line ``\raise1pt\hbox to 0pt{\rule{15pt}{.5pt}\hss}%
\raise3pt\hbox{\rule{15pt}{.5pt}}''
between the node $s_3$ and the edge connecting the nodes $t_1$ and $t_0$
indicates the relation
$s_3(t_1 t_0)s_3(t_1 t_0) = (t_1 t_0)s_3(t_1 t_0)s_3$.

Setting $r=2$ results in the classical presentation of type $\mathbf{I}_2(e)$,
and the subpresentation on the generators $s_3,\ldots, s_r$ is
the classical presentation of type ${\mathbf{A}}_{r-2}$ (see next subsection).
Indeed, the subpresentation on the generators $t_i, s_3,\ldots, s_r$
is the classical presentation of type ${\mathbf{A}}_{r-1}$.

\begin{figure}
$$\begin{array}{c}
\begin{xy}
(10,8)  *++={\rule{0pt}{4pt}} *\frm{o};
(10,-4) *++={\rule{0pt}{4pt}} *\frm{o} **@{-} ?(0.5) *!/^2mm/{e};
(10,2)  *++={} ;
(20,2)  *++={\rule{0pt}{4pt}} *\frm{o} **@{=};
(10,8)  *++={\rule{0pt}{4pt}} *\frm{o} ;
(20,2)  *++={\rule{0pt}{4pt}} *\frm{o} **@{-};
(10,-4) *++={\rule{0pt}{4pt}} *\frm{o} ;
(20,2)  *++={\rule{0pt}{4pt}} *\frm{o} **@{-};
(30,2)  *++={\rule{0pt}{4pt}} *\frm{o} **@{-};
(40,2)  *++={\rule{0pt}{4pt}} *\frm{o} **@{-};
(50,2)  *++={\dots} **@{-} ;
(60,2)  *++={\rule{0pt}{4pt}} *\frm{o} **@{-};
(6, 8)   *++={t_1} ;
(6,-4)  *++={t_0} ;
(23,-1) *++={s_3};
(33,-1) *++={s_4};
(43,-1) *++={s_5};
(63,-1) *++={s_r}
\end{xy}\\[4ex]
\mbox{(a) Brou\'e-Malle-Rouquier diagram for $B(e,e,r)$}\\[2mm]
\begin{xy}
(-1.7, 0) *{\ellipse(9,15){}};
(4.2, 12) *++={\rule{0pt}{4pt}} *\frm{o};
    (20,0) *++={\rule{0pt}{4pt}} *\frm{o} **@{-};
(4.2,-12) *++={\rule{0pt}{4pt}} *\frm{o};
    (20,0) *++={\rule{0pt}{4pt}} *\frm{o} **@{-};
(7.1, 4) *++={\rule{0pt}{4pt}} *\frm{o};
    (20,0) *++={\rule{0pt}{4pt}} *\frm{o} **@{-};
(7.1,-4) *++={\rule{0pt}{4pt}} *\frm{o};
    (20,0) *++={\rule{0pt}{4pt}} *\frm{o} **@{-};
(30, 0) *++={\rule{0pt}{4pt}} *\frm{o} **@{-};
(40, 0) *++={\rule{0pt}{4pt}} *\frm{o} **@{-};
(50, 0) *++={\dots}  **@{-};
(60, 0) *++={\rule{0pt}{4pt}} *\frm{o} **@{-};
(-9,0) *++={2};
(9, 12) *++={t_1};
(10,-15) *++={t_{e-2}};
(11, 1) *++={t_0};
(11,-7.5) *++={t_{e-1}};
(23,-3) *++={s_3};
(33,-3) *++={s_4};
(43,-3) *++={s_5};
(63,-3) *++={s_r};
(0, 17.5) *++={\cdot};
(1.1, 17) *++={\cdot};
(2, 16) *++={\cdot};
(0, -17.5) *++={\cdot};
(1.1, -17) *++={\cdot};
(2, -16) *++={\cdot};
\end{xy}\\[4ex]
\mbox{(b) Corran-Picantin diagram for $B(e,e,r)$}
\end{array}
$$
\caption{Diagrams for $B(e,e,r)$}
\label{fig:Beer}
\end{figure}
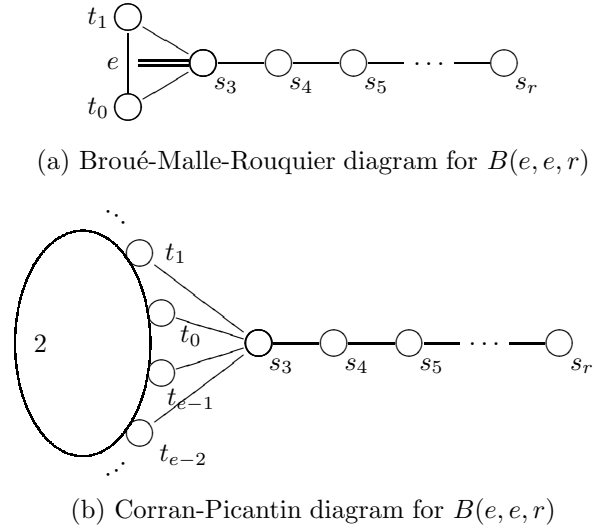

\subsubsection{Corran-Picantin presentation of $B(e,e,r)$}

The first author and Picantin~\cite{CP11} obtained the following
presentation for the braid group $B(e,e,r)$,
the generators of which are \emph{braid reflections} and which
gives rise to a Garside structure.

\begin{theorem}[\cite{CP11}]\label{beerpresthm}
The braid group $B(e,e,r)$ for $e,r\ge 2$ has the following presentation:
\begin{itemize}
\item Generators:  $T_e\cup S$
    where $T_e=\{t_i\mid i \in {\mathbb Z}/e \}$ and $S = \{s_3,\ldots,s_r\}$;
\item Relations: the usual braid relations on $S$, along with\\
$\begin{array}{lll}
(Q_1) & t_i t_{i-1} = t_j t_{j-1}
    \quad \mbox{for $i, j \in {\mathbb Z}/e$},\\
(Q_2) & s_3 t_i s_3 =t_i s_3 t_i
    \quad \mbox{for $i \in {\mathbb Z}/e$},\\
(Q_3) & s_j t_i  =t_i s_j
    \quad \mbox{for $i \in {\mathbb Z}/e$ and $4 \leq j \leq r$}.
\end{array}$
\end{itemize}
Furthermore, adding the relations $a^2=1$ for all generators $a$
gives a presentation of
the imprimitive reflection group $G(e,e,r)$, where the generators
are all reflections.
\end{theorem}


Denote by $\overline{\phantom{x}}$ the natural map
$B(e,e,r) \twoheadrightarrow G(e,e,r)$.
The generating reflections of $G(e,e,r)$
in this new presentation are the following $r\times r$ matrices:
$$\overline{t_i} =\left(\begin{array}{c|c}
 \begin{array}{rl}
0 &  \zeta_e^{-i} \\
 \zeta_e^i &  0 \\
\end{array}
& \begin{array}{c} \\[-.5em] 0\\[-.5em]\\  \end{array} \\
\hline
\\[-1em]
0  & ~ {\LARGE I_{r-2}} ~ \\[-1em]
\\
\end{array} \right)
\quad \mbox{ and } \quad
\overline{s_j} = \mbox{ permutation matrix of } (j-1 \ \ j),$$
where $\zeta_e$ is a primitive $e$-th root of unity.

\subsubsection{Diagram of type $(e,e,r)$}

The diagram shown in Figure~\ref{fig:Beer}(b) was proposed in ~\cite{CP11},
as a type $(e,e,r)$-analogy to the Coxeter graphs for the real reflection group case.
In the diagram, there are $e$ nodes labeled $t_i$, $i\in{\mathbb Z}/e$,
that are tangent to a circle labeled 2.
Whenever two nodes $a$ and $b$ are tangent to the circle, there
is a relation of the form
$aa'=bb'$ where $a'$ and $b'$ are
the nodes immediately preceding $a$ and $b$ respectively on the circle.
If two nodes $a$ and $b$ are neither connected by an edge nor tangent to the circle,
then there is a relation of the form $ab=ba$.

Naively, this appears like the \emph{dual} presentation
of $B(\mathbf{I}_2(e))$ (on the generators $t_i$)
combined with the \emph{classical} presentation of $B({\mathbf{A}}_{r-2})$
(on the generators $s_i$).

\medskip

The group of graph automorphisms of the diagram is the cyclic group of order $e$,
which may be generated by the automorphism $\tau$ which rotates the circle
in the positive direction by a turn of $2 \pi/e$.
This sends the node $t_i$ to the node $t_{i+1}$ for $i\in{\mathbb Z}/e$,
and fixes the nodes $s_j$ for $3 \leq j \leq r$.
By the symmetry of the presentation, these diagram automorphisms
give rise to automorphisms
of the braid group $B(e,e,r)$ as well as of the reflection group $G(e,e,r)$.
These automorphisms send (braid) reflections to (braid) reflections.

\subsubsection{Maps between the groups $B(e,e,r)$ for different values of $e$}

Consider a sequence of natural numbers $\{e_i\}_{i\ge 0}$ such that
$e_0=1$ and  $e_i$ divides $e_{i+1}$ for each $i\ge 0$.
For each $i$, define an epimorphism
$$\nu_{e_i}^{e_{i+1}} : B(e_{i+1},e_{i+1},r) \twoheadrightarrow B(e_i,e_i,r)$$
by $t_k \mapsto t_{k \bmod e_i} $ and $s_j \mapsto s_j$. Then
$$
\nu_{e_{i}}^{e_{i+1}} \circ \nu_{e_{i+1}}^{e_{i+2}} = \nu_{e_{i}}^{e_{i+2}}
\qquad
\mbox{ for all } i \geq 0. $$

\subsection{Braid groups of types ${\mathbf{A}}$, ${\mathbf{B}}$ and $\widetilde{{\mathbf{A}}}$ and geometric braids}
\label{ssec:GeomBr}

The group $B(1,1,r+1)$ for $r\ge 1$ is precisely the Artin braid group $B_{r+1}$
on $r+1$ strings---also known as the braid group of type ${\mathbf{A}}_r$,
denoted $B({\mathbf{A}}_r)$---and possesses the following presentation.
\begin{equation*}
B({\mathbf{A}}_r)=B_{r+1}=\left\langle \sigma_1,\ldots,\sigma_r \biggm|
\begin{array}{ll}
\sigma_i\sigma_j=\sigma_j\sigma_i & \mbox{for $|i-j| > 1$} \\
\sigma_i\sigma_{i+1}\sigma_i=\sigma_{i+1}\sigma_i\sigma_{i+1}
& \mbox{for $i=1,\ldots,r-1$}
\end{array}
\right\rangle
\end{equation*}

The group $B(2,1,r)$ (or indeed $B(d,1,r)$ for any $d \geq 2$) for $r\ge 2$ is usually
called the braid group of type ${\mathbf{B}}_r$, denoted $B({\mathbf{B}}_r)$,
and has the following presentation.
\begin{equation}\label{pres:B_r}
B({\mathbf{B}}_r)=\left< {b_1,\ldots,b_r \Biggm|
{\begin{array}{ll}
b_ib_j=b_jb_i & \mbox{for $|i-j| > 1$} \\
b_ib_{i+1}b_i=b_{i+1}b_ib_{i+1}
& \mbox{for $1< i< r$}\\
b_1b_2b_1b_2=b_2b_1b_2b_1
\end{array}}}
\right>
\end{equation}

See Figure~\ref{fig:typeAB} for the diagrams for the above presentations.
The double edge between $b_1$ and $b_2$ encodes
the relation $b_1b_2b_1b_2=b_2b_1b_2b_1$.
The braid group $B({\mathbf{B}}_r)$ is a subgroup of $B_{r+1}$
of index $r+1$ under the identification
$b_1=\sigma_1^2$ and $b_i=\sigma_i$ for $i=2,\ldots,r$.
See Figure~\ref{fig:braids}.

\begin{figure}
$$\begin{array}{cc}
\begin{xy}
(10,0) *++={\rule{0pt}{4pt}} *\frm{o};
(20,0) *++={\rule{0pt}{4pt}} *\frm{o}**@{-};
(30,0) *++={\rule{0pt}{4pt}} *\frm{o}**@{-};
(40, 0) *++={\dots}  **@{-};
(50, 0) *++={\rule{0pt}{4pt}} *\frm{o} **@{-};
(13,-3) *++={\sigma_1};
(23,-3) *++={\sigma_2};
(33,-3) *++={\sigma_3};
(53,-3) *++={\sigma_r}
\end{xy} &
\qquad
\begin{xy}
(10,0) *++={\rule{0pt}{4pt}} *\frm{o};
(20,0) *++={\rule{0pt}{4pt}} *\frm{o}**@{=};
(30, 0) *++={\rule{0pt}{4pt}} *\frm{o} **@{-};
(40, 0) *++={\dots}  **@{-};
(50, 0) *++={\rule{0pt}{4pt}} *\frm{o} **@{-};
(14,-3) *++={b_1};
(24,-3) *++={b_2};
(34,-3) *++={b_3};
(54,-3) *++={b_r}
\end{xy}\\[15pt]
\mbox{(a) Diagram for $B({\mathbf{A}}_r)$, $r\ge 1$} &
\qquad\mbox{(b) Diagram for $B({\mathbf{B}}_r)$, $r\ge 2$}
\end{array}
$$
\vskip-.5\baselineskip
\caption{Coxeter graphs of $B({\mathbf{A}}_r)$ and $B({\mathbf{B}}_r)$}
\label{fig:typeAB}
\end{figure}
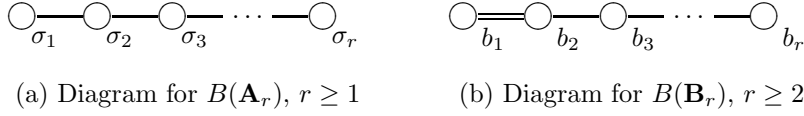

\begin{figure}
\begin{tabular}{*7c}
\includegraphics[scale=.8]{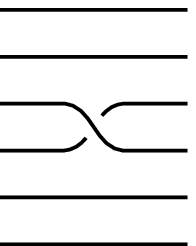}&&
\includegraphics[scale=.8]{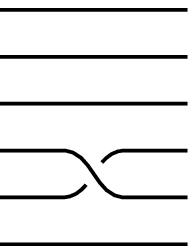}&&
\includegraphics[scale=.8]{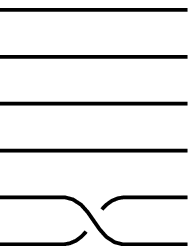}&&
\includegraphics[scale=.8]{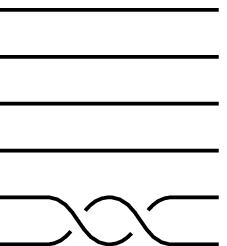}\\
(a) $\sigma_3$ &&
(b) $\sigma_2$ &&
(c) $\sigma_1$ &&
(d) $b_1=\sigma_1^2$
\end{tabular}
\caption{Braid pictures for $\sigma_3$, $\sigma_2$, $\sigma_1$ and $b_1=\sigma_1^2$ in $B_6$}
\label{fig:braids}
\end{figure}

The braid group of type ${\widetilde{\mathbf A}_{r-1}}
$ for $r\ge 3$, denoted $B({\widetilde{\mathbf A}_{r-1}}
)$, is usually described
by the Coxeter graph in Figure~\ref{fig:Atilde_Coxeter}.
This diagram defines a presentation
\begin{equation}\label{pres:A_til_r}
B(\widetilde{\mathbf{A}}_{r-1})=\left\langle  s_1,\ldots, s_r \biggm|
\begin{array}{ll}
s_i s_j=s_j s_i & \mbox{for $i-j \not\equiv \pm 1 \mod{r}$} \\
s_i s_{j}s_i=s_{j}s_i s_{j}
& \mbox{for $i-j \equiv \pm 1 \mod{r}$}
\end{array}
\right\rangle.
\end{equation}
On adding the relations $s_i^2 = 1$ for all $i$,
a presentation for the affine reflection group of type ${\widetilde{\mathbf A}_{r-1}}
$
is obtained, where the generators are (real affine) reflections.

\begin{figure}[t]
$$\begin{xy}
(60, 0) *++={\rule{0pt}{4pt}} *\frm{o};
(35,15) *++={\rule{0pt}{4pt}} *\frm{o} **@{-};
(10, 0) *++={\rule{0pt}{4pt}} *\frm{o} **@{-};
(20, 0) *++={\rule{0pt}{4pt}} *\frm{o} **@{-};
(30, 0) *++={\rule{0pt}{4pt}} *\frm{o} **@{-};
(40, 0) *++={\dots}  **@{-};
(50, 0) *++={\rule{0pt}{4pt}} *\frm{o} **@{-};
(60, 0) *++={\rule{0pt}{4pt}} *\frm{o} **@{-};
(13,-3) *++={s_2};
(23,-3) *++={s_3};
(33,-3) *++={s_4};
(55,-3) *++={s_{r-1}};
(63,-3) *++={s_r};
(35,11) *++={s_1}
\end{xy}$$
\vskip-1.2\baselineskip
\caption{Coxeter graph of $B({\widetilde{\mathbf A}_{r-1}}
)$, $r\ge 3$}
\label{fig:Atilde_Coxeter}
\end{figure}
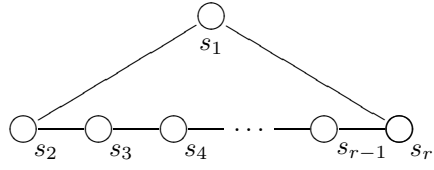

\begin{definition}
\begin{enumerate}
\item
A braid $g\in B_{r+1}$ is said to be \emph{$i$-pure} for $1\le i\le r+1$ if $\pi_g(i)=i$,
where $\pi_g$ denotes the induced permutation of $g$.
Let $B_{r+1,1}$ denote the subgroup of $B_{r+1}$
consisting of 1-pure braids, hence it is generated
by $\{\sigma_1^2,\sigma_2,\sigma_3,\ldots,\sigma_r\}$.

\item
Let $P$ be a subset of\/ $\{1,\ldots,r+1\}$. An $(r+1)$-braid $g$ is said to be
\emph{$P$-pure} if $g$ is $i$-pure for each $i\in P$.
It is said to be \emph{$P$-straight} if it is  $P$-pure and it becomes trivial
when we remove all the $i$-th strings for $i\not\in P$.
\item
The homomorphism $\operatorname{wd}:B_{r+1,1}\to {\mathbb Z}$ measures
the \emph{winding number} around the first string of the other strings.
In particular, $\operatorname{wd}(\sigma_1^2)=1$ and $\operatorname{wd}(\sigma_i)=0$ for all $2\le i\le r$.
\end{enumerate}
\end{definition}

\begin{figure}
$$
\includegraphics{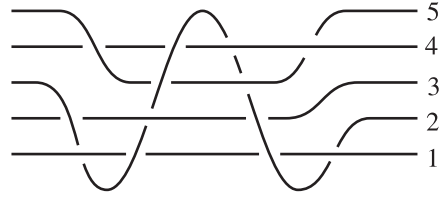}
$$
\caption{This braid is $\{1,4,5\}$-pure, $\{1,4\}$-straight and
$\{1,5\}$-straight, and has winding number $0$.}\label{fig:p-pure}
\end{figure}

For example, the braid in Figure~\ref{fig:p-pure}
is $\{1,4,5\}$-pure, $\{1,4\}$-straight and $\{1,5\}$-straight,
and has winding number 0.
Notice that if $|P|=1$, then a braid is $P$-pure if and only if
it is $P$-straight,
and that if $P=\{1,\ldots,r+1\}$, then $P$-pure braids are
nothing more than pure braids in the usual sense
and the identity braid is the only $P$-straight braid.

\smallskip

It is well known that the braid group $B({\mathbf{B}}_r)$ is isomorphic to
the group of $r$-braids on an annulus and the braid group $B(\widetilde{\mathbf{A}}_{r-1})$
is isomorphic to the subgroup consisting of such braids
with zero winding number~\cite{Cri99, All02, CC05, BM07}.
Equivalently, $B({\mathbf{B}}_r)$ is isomorphic to the group
of 1-pure $(r+1)$-braids on a disk
and $B(\widetilde{\mathbf{A}}_{r-1})$ is isomorphic to the subgroup
consisting of such braids
with zero winding number around the first string.
In this paper, we use the latter isomorphisms
\begin{align*}
B({\mathbf{B}}_r) &\cong  B_{r+1,1},  \\
B(\widetilde {\mathbf{A}}_{r-1}) &\cong
\{\, g\in B_{r+1,1}\mid \operatorname{wd}(g)=0 \, \}.
\end{align*}

The embedding of $B({\widetilde{\mathbf A}_{r-1}}
)$ into $B({\mathbf{B}}_r)$ can be made explicit
in the following way.
Consider a regular $r$-gon in the interior of a disk whose edges are
labeled $E_1,\ldots,E_r$ in clockwise order as in Figure~\ref{fig:br-pic-1}(a).
Suppose that one puncture is at the center and $r$ punctures
are on the vertices of the $r$-gon.
As mapping classes, the generator $s_i$ of $B(\widetilde{{\mathbf{A}}}_{r-1})$ is represented
by a positive half Dehn twist along $E_i$.
The configuration in Figure~\ref{fig:br-pic-1}(a) is equivalent
to that in Figure~\ref{fig:br-pic-1}(b),
from which the generators $s_i$ are expressed as words in $\sigma_j$'s as follows.
\begin{align*}
s_1&=(\sigma_r\cdots \sigma_3)
    (\sigma_1^{-2}\sigma_2\sigma_1^{2})
    (\sigma_3^{-1}\cdots \sigma_r^{-1}), \\
s_j&=\sigma_j\quad\text{for } 2\le j\le r.
\end{align*}
See Figure~\ref{fig:br-pic-1}(c) for the braid picture of the generator $s_1$.

\begin{figure}
$$\small\begin{array}{ccc}
\begin{array}[b]{c}{\begin{xy}
(0,0) *{\ellipse(18,18){}};
(0,18)*{}; (0,-18)*{}; (18,0)*{}; (-18,0)*{};
(0,0) *{\bullet};
(0, 12);
(10, 6) *{\bullet}**@{-}?(0.5) *!/_3mm/{E_1};
(10, 6);
(10,-6) *{\bullet}**@{-}?(0.5) *!/_3mm/{E_2};
(10,-6);
(0, -12) *{\bullet}**@{-}?(0.5) *!/_3mm/{E_3};
(0, -12);
(-10,-6) *{\bullet}**@{-}?(0.5) *!/_3mm/{E_4};
(-10,-6);
(-10, 6) *{\bullet}**@{-}?(0.5) *!/_3mm/{E_5};
(-10, 6);
(0, 12) *{\bullet}**@{-} ?(0.5) *!/_3mm/{E_6};
\end{xy}}\\[1mm]
\mbox{(a)\rule{0pt}{15pt}}
\end{array}
&
\begin{array}[b]{c}
{\begin{xy}
(0,0) *{\ellipse(18,18){}};
(0,18)*{}; (0,-18)*{}; (18,0)*{}; (-18,0)*{};
(0,-12) *{\bullet};
(0,12) *{\bullet};
(0,8) *{\bullet}**@{-}?(0.5) *!/^3mm/{E_6};
(0,8);
(0,4) *{\bullet}**@{-}?(0.5) *!/^3mm/{E_5};
(0,4);
(0,0) *{\bullet}**@{-}?(0.5) *!/^3mm/{E_4};
(0,0);
(0,-4) *{\bullet}**@{-}?(0.5) *!/^3mm/{E_3};
(0,-4);
(0,-8) *{\bullet}**@{-}?(0.5) *!/^3mm/{E_2};
(0,-8); (0,12) **\crv{(-8,-8) & (-8,-16) &  (8,-16) & (8,-8)
  & (8,8) & (8,12)}?(0.6) *!/^3mm/{E_1};
\end{xy}}\\
\mbox{(b)\rule{0pt}{15pt}}
\end{array}
&
\begin{array}[b]{c}
\includegraphics[scale=.9]{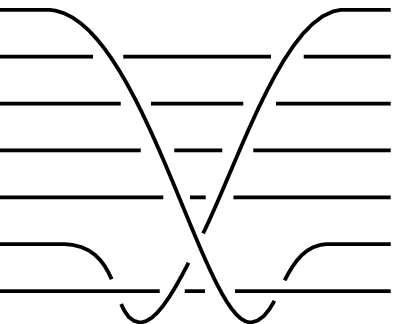}\\[4mm]
\mbox{(c)}
\end{array}
\end{array}
$$
\caption{(a) and (b) are equivalent configurations of punctures and arcs.
As mapping classes, the generator $s_i$ of $B({\widetilde{\mathbf{A}}}_{r-1})$ is represented by
a positive half Dehn twist along $E_i$.
(c) shows a braid picture for the generator $s_1$, which is a 1-pure braid with
winding number zero.
}
\label{fig:br-pic-1}
\end{figure}

In the same way that the free group on two generators
can be considered as the braid group
of a dihedral group $\mathbf{I}_2(\infty)$, the parameter $e$
of the braid group $B(e,e,r)$
can be set to infinity to obtain a group $B(\infty,\infty,r$).
This turns out to be isomorphic to the braid group of type
${\widetilde{\mathbf A}_{r-1}}
$---see the following section and \S\ref{sec:Atilde}
for an extended discussion.

\section{New presentations for the braid groups $B(\infty,\infty,r)$ and $B(de,e,r)$}
\label{sec:Binfty}

By suppressing the relation $(P_1)$ in the Brou\'e-Malle-Rouquier
presentation of  $B(e,e,r)$,
we obtain a group which we will denote by $B(\infty,\infty,r)$.
In other words, $B(\infty,\infty,r)$, $r\ge 2$, has the following presentation
(see Figure~\ref{fig:Binfty}(a) for the diagram for this presentation):

\begin{itemize}
\item Generators: $\{t_0,t_1\} \cup S$
    where $S = \{s_3,\ldots,s_r\}$;
\item Relations: the usual braid relations on $S$, along with\\
$\begin{array}{ll}
(P_2) & s_3 t_i s_3 =t_i s_3 t_i \quad \mbox{for $i = 0,1$}, \\
(P_3) & s_j t_i = t_i s_j \quad \mbox{for $i = 0,1$ and $4\leq j\leq r$},\\
(P_4) & s_3(t_1 t_0)s_3(t_1 t_0) = (t_1 t_0)s_3(t_1 t_0)s_3.
\end{array} $
\end{itemize}

This group was first considered by Shi~\cite{Shi02} as
a generic version of the groups $B(e,e,r)$.
He observed that on adding the relations $a^2=1$
for all the generators $a$, a presentation
was obtained for a group denoted $G(\infty,\infty,r)$.
Denote by $\overline{\phantom{x}}$ the natural map
$B(\infty,\infty,r)\twoheadrightarrow G(\infty,\infty,r)$.
The generating reflections in the presentation of
$G(\infty, \infty, r)$ are the $r\times r$ matrices
$$
\overline{t_i} =\left(\begin{array}{c|c}
\begin{array}{rl}
0 &  x^{-i} \\
x^i &  0 \\
\end{array}
& 0\\
\hline
\\[-1em]
0  & ~ {\LARGE I_{r-2}} ~ \\[-1em]
\\
\end{array} \right)
\quad \mbox{and}\quad \overline{s_j}=\mbox{permutation matrix of $(j-1\ \  j)$},$$
where $x$ is a transcendental number.
The matrices are complex reflections of order 2.

\medskip

\begin{figure}
$$\begin{array}{c}
\xy
(10,8)  *++={\rule{0pt}{4pt}} *\frm{o};
(10,-4) *++={\rule{0pt}{4pt}} *\frm{o} **@{-} ?(0.5) *!/^2.5mm/{\infty};
(10,2)  *++={} ;
(20,2)  *++={\rule{0pt}{4pt}} *\frm{o} **@{=};
(10,8)  *++={\rule{0pt}{4pt}} *\frm{o} ;
(20,2)  *++={\rule{0pt}{4pt}} *\frm{o} **@{-};
(10,-4) *++={\rule{0pt}{4pt}} *\frm{o} ;
(20,2)  *++={\rule{0pt}{4pt}} *\frm{o} **@{-};
(30,2)  *++={\rule{0pt}{4pt}} *\frm{o} **@{-};
(40,2)  *++={\rule{0pt}{4pt}} *\frm{o} **@{-};
(50,2)  *++={\dots} **@{-} ;
(60,2)  *++={\rule{0pt}{4pt}} *\frm{o} **@{-};
(6,8)   *++={t_1} ;
(6,-4)  *++={t_0} ;
(23,-1) *++={s_3};
(33,-1) *++={s_4};
(43,-1) *++={s_5};
(63,-1) *++={s_r}
\endxy\\
\mbox{(a) Diagram for Shi's presentation of $B(\infty,\infty,r)$}\\[2mm]
\begin{xy}
(2.2,-25) *++={\vdots};
(2.2, 25) *++={\vdots} **@{-};
(4, 16) *++={\rule{0pt}{4pt}} *\frm{o};
    (20,0) *++={\rule{0pt}{4pt}} *\frm{o} **@{-};
(4,-16) *++={\rule{0pt}{4pt}} *\frm{o};
    (20,0) *++={\rule{0pt}{4pt}} *\frm{o} **@{-};
(4, 8) *++={\rule{0pt}{4pt}} *\frm{o};
    (20,0) *++={\rule{0pt}{4pt}} *\frm{o} **@{-};
(4,-8) *++={\rule{0pt}{4pt}} *\frm{o};
    (20,0) *++={\rule{0pt}{4pt}} *\frm{o} **@{-};
(4, 0) *++={\rule{0pt}{4pt}} *\frm{o};
    (20,0) *++={\rule{0pt}{4pt}} *\frm{o} **@{-};
(30, 0) *++={\rule{0pt}{4pt}} *\frm{o} **@{-};
(40, 0) *++={\rule{0pt}{4pt}} *\frm{o} **@{-};
(50, 0) *++={\dots}  **@{-};
(60, 0) *++={\rule{0pt}{4pt}} *\frm{o} **@{-};
(0,-18) *++={2};
(6,12) *++={t_2};
(6, 4) *++={t_1};
(6,-4) *++={t_0};
(7,-12)*++={t_{-1}};
(7,-20)*++={t_{-2}};
(23,-3) *++={s_3};
(33,-3) *++={s_4};
(43,-3) *++={s_5};
(63,-3) *++={s_r};
\end{xy}\\
\mbox{(b) Diagram for the new presentation of $B(\infty,\infty,r)$}
\end{array}
$$
\caption{Diagrams for $B(\infty,\infty,r)$}
\label{fig:Binfty}
\end{figure}

\subsection{New presentation of $B(\infty,\infty,r)$}

For the braid group $B(\infty,\infty,r)$, we introduce a new presentation,
analogous to the Corran-Picantin presentation of $B(e,e,r)$.

\begin{theorem}\label{binftyrpresthm}
The braid group $B(\infty,\infty,r)$ for $r\ge 2$ has the following presentation:
\begin{itemize}
\item Generators: $T\cup S$ where $T = \{t_i\mid i \in {\mathbb Z} \}$
    and $S = \{s_3,\ldots,s_r\}$;
\item Relations: the usual braid relations on $S$, along with\\
$\begin{array}{ll}
(Q_1) & t_i t_{i-1} = t_j t_{j-1}
    \quad \mbox{for $i, j \in {\mathbb Z}$},\\
(Q_2) & s_3 t_i s_3 =t_i s_3 t_i \quad\mbox{for $i \in {\mathbb Z}$},\\
(Q_3) & s_j t_i  =t_i s_j  \quad \mbox{for $i\in{\mathbb Z}$ and $4\leq j \leq r$}.
\end{array}$
\end{itemize}
Furthermore, adding the relations $a^2=1$ for all generators $a$
gives a presentation of
the group $G(\infty,\infty,r)$, where the generators are all reflections.
\end{theorem}
The proof of the above theorem is the same as that for $B(e,e,r)$ in \cite{CP11}.
However, we give the proof in \S\ref{ssec:Pf_Binfty} for completeness.

\subsubsection{Diagram of type $(\infty,\infty,r)$}

In the obvious generalization of the type $(e,e,r)$ case,
we propose the diagram shown in Figure~\ref{fig:Binfty}(b)
as a type $(\infty,\infty,r)$ diagram.
Notice that the diagram for the new presentation is obtained from
Shi's diagram by changing
{\Small\begin{xy}
(0,2) *++={\phantom{.}} *\frm{o};
(10,2) *++={\phantom{.}} *\frm{o}**@{-}  ?(0.5) *!/_2mm/{\infty};
(0,-2) *++={t_0};
(10,-2) *++={t_1}
\end{xy}}
to the dual diagram for $F_2$.

Let $\tau$ denote the graph automorphism
$$
t_i\mapsto t_{i+1}\quad \mbox{for $i\in{\mathbb Z}$},\qquad
s_j\mapsto s_j\quad \mbox{for $3\le j\le r$}.
$$
This gives rise to automorphisms of $B(\infty,\infty,r)$ and $G(\infty,\infty,r)$
which send (braid) reflections to (braid) reflections.

\subsubsection{Maps between $B(e,e,r)$ and $B(\infty,\infty,r)$}

Let $\nu_e:B(\infty,\infty,r) \to B(e,e,r)$ be the epimorphism
which sends $t_i$ to $t_{i \bmod e}$ for $i\in{\mathbb Z}$ and
sends $s_j$ to $s_j$ for $3 \leq j \leq r$.
Once again, consider a sequence of natural numbers $e_i$ such that
$e_0 = 1$ and
$e_i$ divides $e_{i+1}$ for each $i\ge 0$.
Then
$$
\nu_{e_{i}}^{e_{i+1}} \circ \nu_{e_{i+1}}^{\phantom{e_i}}
= \nu_{e_{i}}^{\phantom{e_i}}
\quad\mbox{and}\quad
\nu_{e_{i}}^{e_{i+1}} \circ \nu_{e_{i+1}}^{e_{i+2}} = \nu_{e_{i}}^{e_{i+2}}
\qquad
\mbox{for all $i \geq 0$}.
$$
The inverse limit of the sequence
$$ \xymatrix{
\cdots \ar @{->>}[r]&
B(e_{i+1},e_{i+1},r) \ar @{->>}[r]^{\quad\nu_{e_i}^{e_{i+1}}}
& B(e_i,e_i,r) \ar @{->>}[r]^{{}\qquad\nu_{e_{i-1}}^{e_i}\quad}
& \cdots \ar @{->>}[r]^{\nu_{e_1}^{e_2}\quad}
& B(e_1,e_1,r) \ar @{->>}[r]^{\nu_1^{e_1}}
&B(1,1,r)
} $$
is however \emph{not} $B(\infty,\infty,r)$, but rather a group we
denote by $B(\hat{\mathbb Z},\hat{\mathbb Z},r)$,
where $\hat{{\mathbb Z}}$ is the profinite completion of $\mathbb Z$~\cite{Shi02}.

\subsubsection{Proof of Theorem~\ref{binftyrpresthm}}
\label{ssec:Pf_Binfty}

For completeness, we include a proof of Theorem~\ref{binftyrpresthm}.
Add new generators $\{t_i\mid i\in{\mathbb Z}\setminus\{0,1\}\}$ to Shi's presentation
which are defined inductively by
$$t_i=
\begin{cases}
t_{i-1}t_{i-2}t_{i-1}^{-1} & \mbox{for $i\ge  2$},\\
t_{i+1}^{-1}t_{i+2}t_{i+1} & \mbox{for $i\le -1$}.
\end{cases}
$$
The above relations are the same as
$t_it_{i-1}=t_{i-1}t_{i-2}$ for all $i\in{\mathbb Z}$,
which is the same as the relation $(Q_1)$.
Therefore $B(\infty,\infty,r)$ has the following presentation:
\begin{itemize}
\item Generators: $T\cup S$ where $T = \{t_i\mid i \in {\mathbb Z} \}$
    and $S = \{s_3,\ldots,s_r\}$;
\item Relations: the usual braid relations on $S$, along with\\
$\begin{array}{ll}
(Q_1) & t_i t_{i-1} = t_j t_{j-1}
    \quad \mbox{for $i, j \in {\mathbb Z}$},\\
(P_2) & s_3 t_i s_3 =t_i s_3 t_i \quad \mbox{for $i = 0,1$}, \\
(P_3) & s_j t_i = t_i s_j \quad \mbox{for $i = 0,1$ and $4\leq j\leq r$},\\
(P_4) & s_3(t_1 t_0)s_3(t_1 t_0) = (t_1 t_0)s_3(t_1 t_0)s_3.
\end{array}$
\end{itemize}

\begin{claim}{Claim 1}
Assuming $(Q_1)$, the relation $(Q_3)$ is equivalent to $(P_3)$.
\end{claim}

\begin{proof}[Proof of Claim 1]
$(P_3)$ is a special case of $(Q_3)$.
Assuming $(Q_1)$, every $t_i$ is represented by a word in $\{t_0, t_1\}$:
for $m\in{\mathbb Z}$ and $k=0,1$,
$$
t_{2m+k}=(t_1t_0)^m t_k (t_1t_0)^{-m}.
$$
If we assume $(P_3)$, then $s_j$ ($4\le j\le r$) commutes with $t_0$ and $t_1$,
hence $s_j$ commutes with $t_i$ for any $i\in {\mathbb Z}$, which is the relation $(Q_3)$.
Therefore $(P_3)$ implies $(Q_3)$.
\end{proof}

\begin{claim}{Claim 2}
Assuming $(Q_1)$, the relation $(Q_2)$ is equivalent to $(P_2)+(P_4)$.
\end{claim}
\begin{proof}[Proof of Claim 2]
Suppose that $(Q_2)$ holds, that is, $s_3 t_i s_3 =t_i s_3 t_i$ for all $i \in {\mathbb Z}$.
Notice that $(P_2)$ is a special case of $(Q_2)$.
The following formula shows that $(P_4)$ also holds,
where relations are applied to the underlined subwords.
\begin{align*}
s_3\underline{t_1t_0}s_3t_1t_0
 &\stackrel{Q_1}{=}s_3t_2\underline{t_1s_3t_1}t_0
 \stackrel{Q_2}{=} \underline{s_3t_2s_3}t_1s_3t_0
 \stackrel{Q_2}{=} t_2s_3\underline{t_2t_1}s_3t_0\\
&\stackrel{Q_1}{=} t_2s_3t_1\underline{t_0s_3t_0}
 \stackrel{Q_2}{=} t_2\underline{s_3t_1s_3}t_0s_3
 \stackrel{Q_2}{=}\underline{t_2t_1}s_3t_1t_0s_3
 \stackrel{Q_1}{=} t_1t_0s_3t_1t_0s_3.
\end{align*}

Conversely, suppose that both $(P_2)$ and $(P_4)$ hold.
By $(P_2)$, we know that $(Q_2)$ holds for $i=0,1$.
Assume that $(Q_2)$ holds for $i=k,k+1$, which we denote by
$(Q_{2,k})$ and $(Q_{2,k+1})$, respectively. Since
\begin{align*}
s_3t_{k+2} \underline{s_3t_{k+1}s_3} t_k
&\stackrel{Q_{2,k+1}}{=} s_3 \underline{t_{k+2}t_{k+1}} s_3 \underline{t_{k+1}t_k}
 \stackrel{Q_1}{=} s_3t_1t_0s_3t_1t_0
 \stackrel{P_4}{=} \underline{t_1t_0} s_3 \underline{t_1t_0} s_3 \\
& \stackrel{Q_1}{=} t_{k+2} \underline{t_{k+1}s_3t_{k+1}} t_ks_3
 \stackrel{Q_{2,k+1}}{=} t_{k+2}s_3t_{k+1} \underline{s_3t_ks_3}\\
&\stackrel{Q_{2,k}}{=}   t_{k+2}s_3 \underline{t_{k+1}t_k} s_3t_k
 \stackrel{Q_1}{=}       t_{k+2}s_3t_{k+2}t_{k+1}s_3t_k,
\end{align*}
we have $s_3t_{k+2} s_3=t_{k+2}s_3t_{k+2}$ by canceling
$t_{k+1}s_3t_k$ from the right.
Hence $(Q_2)$ holds for $i=k+2$.
Similarly, we can show that $(Q_2)$ holds also for $i=k-1$.
By induction, we conclude that $(Q_2)$ holds for all $i\in{\mathbb Z}$.
\end{proof}

From Claims 1 and 2, the new presentation for $B(\infty,\infty,r)$ is correct.

The new presentation has the same generators as the original,
as well as some conjugates of the originals.
Since it is the case for Shi's presentation,
adding the relations $a^2=1$ for all generators $a$ in the
new presentation gives a presentation
of the reflection group $G(\infty,\infty,r)$.
Since conjugates of reflections are reflections,
the generators of this presentation are all reflections.

\subsection{New presentation of $B(de,e,r)$}
\label{sec:Bdeer_pres_infty}

Brou\'e, Malle and Rouquier~\cite{BMR98} introduced the following
presentation of $B(de,e,r)$:

\begin{itemize}
\item Generators: $\{z \} \cup \{t_0,t_1\} \cup S$ where $S=\{s_j\mid 3\le j\le r\}$;
\item Relations: the usual braid relations on $S$, along with\\
$\begin{array}{lll}
(R_1) & z t_1 t_0 = t_1 t_0 z, \\
(R_2) & z \langle t_1 t_0\rangle^e = t_0 z \langle t_1 t_0\rangle^{e-1}, \\
(R_3) & z s_j =s_j z \quad \mbox{for $3 \leq j \leq r$},  \\
(R_4) & s_3 t_i s_3 =t_i s_3 t_i \quad\mbox{for $i = 0,1$},\\
(R_5) & s_3(t_1 t_0)s_3(t_1 t_0) = (t_1 t_0)s_3(t_1 t_0)s_3,\\
(R_6) & s_jt_i=t_is_j \quad\mbox{for $i=0,1$ and $4\le j\le r$}.
\end{array}$
\end{itemize}

Furthermore, a presentation for the reflection group $G(de, e, r)$ is obtained
by adding the relations $z^d=1$, $t_0^2=t_1^2=1$ and $s_j^2=1$ for $3\le j\le r$,
and the generators are then all reflections.

This presentation is usually illustrated by the diagram in
Figure~\ref{fig:Bdeer}(a) (note the similarity to the diagram
for $B(e,e,r)$ in Figure~\ref{fig:Beer}(a)).
Adding the relation $z=1$ to the above presentation gives
the BMR presentation for $B(e,e,r)$, thus defining an epimorphism
from $B(de,e,r)$ to $B(e,e,r)$.

\begin{figure}
$$\begin{array}{c}
{\begin{xy}
(10,9.9); (10,-6) **\crv{(3,10.5) & (-4,2) & (3,-6.5)} ?(0.65) *!/^5mm/{e+1};
(-2.5,2) *++={\rule{0pt}{4pt}} *\frm{o};
(-6,2)   *++={z} ;
(10,8)  *++={\rule{0pt}{4pt}} *\frm{o};
(10,-4) *++={\rule{0pt}{4pt}} *\frm{o};
(10,2)  *++={} ;
(20,2)  *++={\rule{0pt}{4pt}} *\frm{o} **@{=};
(10,8)  *++={\rule{0pt}{4pt}} *\frm{o} ;
(20,2)  *++={\rule{0pt}{4pt}} *\frm{o} **@{-};
(10,-4) *++={\rule{0pt}{4pt}} *\frm{o} ;
(20,2)  *++={\rule{0pt}{4pt}} *\frm{o} **@{-};
(30,2)  *++={\rule{0pt}{4pt}} *\frm{o} **@{-};
(40,2)  *++={\rule{0pt}{4pt}} *\frm{o} **@{-};
(50,2)  *++={\dots} **@{-} ;
(60,2)  *++={\rule{0pt}{4pt}} *\frm{o} **@{-};
(13,11)   *++={t_1} ;
(13,-7)  *++={t_0} ;
(23,-1) *++={s_3};
(33,-1) *++={s_4};
(43,-1) *++={s_5};
(63,-1) *++={s_r}
\end{xy}}\\
\mbox{(a) Brou\'e-Malle-Rouquier diagram for $B(de,e,r)$}\\
\begin{xy}
(-2,8.7); (-2,-8.3) **\crv{(-9,0)} ?(0.8) *!/^2mm/{e} ?>*@{>};
(-7.5,0) *++={\rule{0pt}{4pt}} *\frm{o};
(-11,0)   *++={z} ;
(1.8,-25) *++={\vdots};
(1.8, 25) *++={\vdots} **@{-};
(4, 16) *++={\rule{0pt}{4pt}} *\frm{o};
    (20,0) *++={\rule{0pt}{4pt}} *\frm{o} **@{-};
(4,-16) *++={\rule{0pt}{4pt}} *\frm{o};
    (20,0) *++={\rule{0pt}{4pt}} *\frm{o} **@{-};
(4, 8) *++={\rule{0pt}{4pt}} *\frm{o};
    (20,0) *++={\rule{0pt}{4pt}} *\frm{o} **@{-};
(4,-8) *++={\rule{0pt}{4pt}} *\frm{o};
    (20,0) *++={\rule{0pt}{4pt}} *\frm{o} **@{-};
(4, 0) *++={\rule{0pt}{4pt}} *\frm{o};
    (20,0) *++={\rule{0pt}{4pt}} *\frm{o} **@{-};
(30, 0) *++={\rule{0pt}{4pt}} *\frm{o} **@{-};
(40, 0) *++={\rule{0pt}{4pt}} *\frm{o} **@{-};
(50, 0) *++={\dots}  **@{-};
(60, 0) *++={\rule{0pt}{4pt}} *\frm{o} **@{-};
(0,-18) *++={2};
(6,12) *++={t_2};
(6, 4) *++={t_1};
(6,-4) *++={t_0};
(7,-12)*++={t_{-1}};
(7,-20)*++={t_{-2}};
(23,-3) *++={s_3};
(33,-3) *++={s_4};
(43,-3) *++={s_5};
(63,-3) *++={s_r}
\end{xy}\\
\mbox{(b) Diagram for the new presentation of $B(de,e,r)$}
\end{array}
$$
\caption{Diagrams for $B(de,e,r)$}
\label{fig:Bdeer}
\end{figure}

In the case $e=1$, type $(de,e,r)=(d,1,r)$ is precisely type ${\mathbf{B}}_r$.
The above presentation is claimed in~\cite{BMR98} as valid for $d,e,r\geq2$
(probably to avoid doubling up for the type ${\mathbf{B}}_r$ presentation).
However it is indeed valid in the $e=1$ case as well.
In this case, the BMR relation $(R_2)$ becomes
$t_1=z^{-1}t_0z$, hence $t_1$ is a superfluous generator.
If we remove $t_1$ from the set of generators and replace every
occurrence of $t_1$ in the defining relations with $z^{-1}t_0z$,
then the BMR presentation is reduced to the presentation of type ${\mathbf{B}}_r$,
under the correspondence
$z\mapsto b_1$, $t_0 \mapsto b_2$ and
$s_i \mapsto b_i$ for $3\le i\le r$.

In this article, when we speak of the BMR-presentation,
it will be implicit that $d,r \geq 2$ and $e \geq 1$.

\subsubsection{Semidirect product with $B(\infty,\infty,r)$}

We propose a new presentation for $B(de,e,r)$
which makes clear the decomposition of
$B(de,e,r)$ as a semidirect product of $B(\infty,\infty,r)$
and an infinite cyclic group.
The theorem will be proved in \S\ref{ssec:Pf_Bdeer_pres_infty}.

\begin{theorem}\label{new-pre-deer}
The braid group $B(de,e,r)$ for $d,r\geq 2$ and $e \geq 1$
has the following presentation:
\begin{itemize}
\item Generators: $\{z \} \cup T \cup S$ where
$T=\{t_i\mid i \in {\mathbb Z}\}$ and $S=\{s_j\mid 3\leq j \leq r\}$;
\item Relations: the usual braid relations on $S$, along with\\
$\begin{array}{ll}
(Q_1) & t_i t_{i-1} = t_j t_{j-1}
    \quad \mbox{for $i, j \in {\mathbb Z}$},\\
(Q_2) & s_3 t_i s_3 =t_i s_3 t_i \quad\mbox{for $i \in {\mathbb Z}$},\\
(Q_3) & s_j t_i  =t_i s_j  \quad \mbox{for $i\in{\mathbb Z}$ and $4\leq j \leq r$},\\
(Q_4) & z t_i = t_{i-e} z \quad\mbox{for $i \in {\mathbb Z}$,}\\
(Q_5) & z s_j = s_j z \quad \mbox{for $3 \leq j \leq r$}.
\end{array}$
\end{itemize}
\end{theorem}

Let $C_{\infty}=\langle c\rangle$ be an infinite cyclic group acting on $B(\infty, \infty, r)$
as follows:
$$
c \cdot t_i = t_{i-1}
\quad \mbox{for $i\in{\mathbb Z}$}, \quad
c\cdot s_j = s_j \quad\mbox{for $3\le j\le r$}.
$$

Let $z = c^e$ and $C_\infty^{e} = \langle z\rangle$.
Then the presentation of $B(de, e, r)$ in Theorem~\ref{new-pre-deer}
can be considered as
a presentation of $C_\infty^{e} \ltimes B(\infty,\infty,r)$:
the relations $(Q_1)$, $(Q_2)$ and $(Q_3)$ are the relations of $B(\infty,\infty,r)$
and the relations $(Q_4)$ and $(Q_5)$ describe the $C^e_\infty$
action on $B(\infty,\infty,r)$.
Therefore we obtain the following.

\begin{corollary}
\label{bdeerthm}
The homomorphism
$\psi\colon B(de,e,r) \to C_\infty^e \ltimes B(\infty,\infty,r)$
given by $\psi(z) = z$,
$\psi(t_i) = t_i$ and $\psi(s_j) = s_j$ for $i\in{\mathbb Z}$ and $3\le j\le r$
is an isomorphism.\end{corollary}

We propose the diagram shown in Figure~\ref{fig:Bdeer}(b)
for the new presentation of $B(de,e,r)$.
The diagram looks like the diagram for $B(\infty,\infty,r)$
in Figure~\ref{fig:Binfty}(b).
The action of $z$ is illustrated by a curved arrow labeled $e$,
describing the relation $zt_i=t_{i-e}z$.

\subsubsection{Reflection group $G(de,e,r)$}

As long as $d,d'\ge 2$, $B(de,e,r) \cong B(d'e,e,r)$.
The parameter $d$ only makes an appearance when it comes
to the reflection group $G(de,e,r)$.
As described in~\cite{BMR98}, adding the relations $z^d=1$ and $a^2=1$
for all the other generators $a$ to the Brou\'e-Malle-Rouquier presentation
of $B(de,e,r)$
gives a presentation for the complex reflection group $G(de,e,r)$.
The generators are all reflections, of order $2$ except $z$ which is of order $d$.

The generators of the new presentation of $B(de, e, r)$
are those of the Brou\'e-Malle-Rouquier presentation
together with some conjugates of them.
Thus, as it is the case in~\cite{BMR98},
adding the relations $z^d =1$ and $a^2=1$ for all the other generators $a$
to the new presentation of $B(de,e,r)$ gives rise to a new presentation
for the reflection group $G(de,e,r)$,
where the generators are all reflections.
This is a presentation on an infinite set of generators for a finite group!
In fact, since $z^d = 1$, we have
$$t_{i+de} = z^d t_{i+de} = t_{i} z^d = t_i \quad\mbox{for all $i\in{\mathbb Z}$}.$$
Thus we have the following isomorphism.

\begin{corollary}
\label{thm:Gdeer}
The reflection group $G(de,e,r)$ for $d,r\ge 2$ and $e\ge 1$ is isomorphic to
the semidirect product $C_{d}^e \ltimes G(de,de,r)$,
where $C_{d}^{e}=\langle  z \rangle$ is a cyclic group of order $d$.
Hence $G(de, e, r)$ has the following presentation:

\begin{itemize}
\item Generators: $\{z \} \cup T_{de} \cup S$
    where $T_{de}=\{t_i \mid i \in {\mathbb Z}/de\}$ and $S=\{s_j \mid 3\leq j \leq r\}$;
\item Relations: all the relations of\/ $G(de,de,r)$ in Theorem~\ref{beerpresthm},
along with
\begin{itemize}
\item
the relations $z t_i = t_{i-e} z$ and $z s_j = s_j z$
for $i\in{\mathbb Z}/de$ and $3\le j\le r$
describing the semidirect product action,

\item
the relations $z^d=1$, $t_i^2 = 1$ and $s_j^2 = 1$ for $i\in{\mathbb Z}/de$ and $3\le j\le r$
describing the order of the generating reflections.
\end{itemize}
\end{itemize}
\end{corollary}

In this presentation of $G(de, e, r)$, the generators can be represented by
the following $r\times r$ matrices:
$$
\overline{t_i} =\left(\begin{array}{c|c}
 \begin{array}{rl}
0 &  \zeta_{de}^{-i} \\
 \zeta_{de}^i &  0 \\
\end{array}
& \begin{array}{c} \\[-.5em] {\LARGE 0} \\[-.5em] \\ \end{array} \\
\hline
\\[-.7em]
0  & ~ {\LARGE I_{r-2}} ~ \\[-1em]
\\
\end{array} \right),\quad
\begin{array}{l}
\overline{z} = \operatorname{Diag}(\zeta_{de}^e,1,1,\ldots,1),\\[1em]
\overline{s_j} = \mbox{permutation matrix of}\ (j-1\ \ j),
\end{array}
$$
where $\zeta_{de}$ is a primitive $de$-th root of unity.

A diagram for the presentation of $G(de,e,r)$ is in Figure~\ref{fig:Gdeer}.
It is obtained from the diagram for $B(de,e,r)$ in Figure~\ref{fig:Bdeer}(b)
by identifying the node $t_i$ with $t_{i+de}$ for each $i\in{\mathbb Z}$.
In particular, the disc at the left has $de$ nodes on it,
and the action of $z$ twists this disc by $e$ nodes.
The numbers inside the nodes denote the orders of
the generators ($z$ has order $d$, all the others have order $2$).

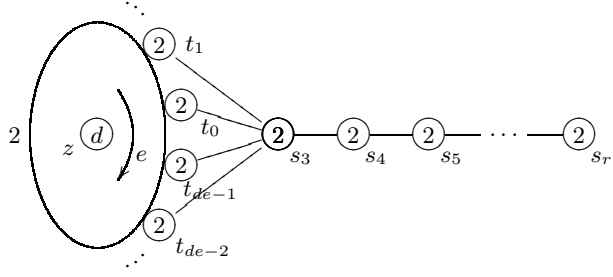
\begin{figure}
\small
$$\begin{xy}
(-2, 0) *{\ellipse(9,15){}};
(4.2, 12) *++={2} *\frm{o};
    (20,0) *++={2} *\frm{o} **@{-};
(4.2,-12) *++={2} *\frm{o};
    (20,0) *++={2} *\frm{o} **@{-};
(7.1, 4) *++={2} *\frm{o};
    (20,0) *++={2} *\frm{o} **@{-};
(7.1,-4) *++={2} *\frm{o};
    (20,0) *++={2} *\frm{o} **@{-};
(30, 0) *++={2} *\frm{o} **@{-};
(40, 0) *++={2} *\frm{o} **@{-};
(50, 0) *++={\dots}  **@{-};
(60, 0) *++={2} *\frm{o} **@{-};
(-15,0) *++={2};
(9, 12) *++={t_1};
(10,-15) *++={t_{de-2}};
(11, 1) *++={t_0};
(11,-7.5) *++={t_{de-1}};
(23,-3) *++={s_3};
(33,-3) *++={s_4};
(43,-3) *++={s_5};
(63,-3) *++={s_r};
(0, 17.5) *++={\cdot};
(1.1, 17) *++={\cdot};
(2, 16) *++={\cdot};
(0, -17.5) *++={\cdot};
(1.1, -17) *++={\cdot};
(2, -16) *++={\cdot};
(-4.1, 0) *++={d} *\frm{o};
(-7.9, -2) *++={z};
(-1.3,6); (-1.3,-6) **\crv{(2.7,0)} ?(0.7) *!/_1.5mm/{e} ?>*@{>};
\end{xy}
$$
\caption{Diagram for the new presentation of $G(de,e,r)$}
\label{fig:Gdeer}
\end{figure}

\subsubsection{Maps between the groups $B(de,e,r)$ for different values of $e$}
\label{ssec:Maps_Bdeer}

Denote by $\iota_e$ the natural embedding:
$$
\iota_e: B(\infty,\infty,r)
\hookrightarrow  C_\infty^{e} \ltimes B(\infty,\infty,r)
\cong B(de,e,r).
$$
Once again, consider a sequence of natural numbers $e_i$
such that $e_0=1$ and $e_i$ divides $e_{i+1}$ for each $i\ge 0$.
Then there are embeddings $C_\infty^{e_{i+1}} \hookrightarrow C_\infty^{e_{i}}$
which maps the generator of $C_\infty^{e_{i+1}}$ to
the $\frac{e_{i+1}}{e_i}$-th power of the generator of $C_\infty^{e_i}$.
These embeddings may be extended to
$
\iota_{e_{i}}^{e_{i+1}} :
C_\infty^{e_{i+1}} \ltimes B(\infty,\infty,r)
\hookrightarrow C_\infty^{e_{i}}\ltimes B(\infty,\infty,r).
$
By Corollary~\ref{bdeerthm}, this map is thus an embedding between the braid groups:
$$
\iota_{e_{i}}^{e_{i+1}} :  B(d e_{i+1}, e_{i+1}, r)
\hookrightarrow B(d e_{i}, e_{i}, r).
$$
Hence we have the following commutative diagram, where the rows are exact.
$$
\xymatrix{
0 \ar[r] & B(\infty,\infty,r) \ar[r]^{\iota_{e_{i+1}}} \ar@{=}[d]
    & B(de_{i+1},e_{i+1},r) \ar[r] \ar@{^{(}->}[d]^{\iota_{e_{i}}^{e_{i+1}}}
    & C^{e_{i+1}}_\infty \ar[r] \ar@{^{(}->}[d] & 0 \\
0 \ar[r] & B(\infty,\infty,r) \ar[r]^{\iota_{e_i}}
    & B(de_i,e_i,r) \ar[r]
    & C^{e_i}_\infty \ar[r] & 0 \\
}$$

Note that
$
\iota_{e_{i}}^{e_{i+1}} \circ \iota_{e_{i+1}} = \iota_{e_{i}}$ and
$\iota_{e_{i}}^{e_{i+1}} \circ \iota_{e_{i+1}}^{e_{i+2}} = \iota_{e_{i}}^{e_{i+2}}$
for all $i \geq 0$,
and that $B(\infty, \infty, r)$ is the inverse limit of the sequence
$$\cdots \hookrightarrow B(d e_{i+1},e_{i+1},r)
\stackrel{\ \iota_{_{e_{i}}}^{^{e_{i+1}}}}{\hookrightarrow}
B(d e_{i},e_{i},r)
\hookrightarrow
\cdots
\stackrel{\ \iota_{_{e_{1}}}^{^{e_{2}}}}{\hookrightarrow}
B(d e_{1},e_{1},r)
\stackrel{\ \iota_{_{1}}^{^{e_{1}}}}{\hookrightarrow}
B(d ,1,r).
$$
We remark that $B(d ,1,r) \cong B({\mathbf{B}}_r)$, the braid group of type ${\mathbf{B}}_r$.
This is discussed in greater length in \S\ref{ssec:Embed_Bdeer_Br}.

\subsubsection{Proof of Theorem~\ref{new-pre-deer}}
\label{ssec:Pf_Bdeer_pres_infty}

Similarly to the proof for the new presentation of $B(\infty,\infty,r)$,
we add new generators $\{t_i\mid i\in{\mathbb Z}\setminus\{0,1\}\}$ to the
Brou\'e-Malle-Rouquier presentation along with the relation
$(Q_1)$ $t_it_{i-1}=t_jt_{j-1}$ for all $i,j\in{\mathbb Z}$.
Then, from the proof of Theorem~\ref{binftyrpresthm},
we know that the relations $(R_4)+(R_5)+(R_6)$ are equivalent
to $(Q_2)+(Q_3)$. The relation $(R_3)$ is identical to $(Q_5)$.
Therefore $B(de,e,r)$ has the following presentation.

\begin{itemize}
\item Generators: $\{z \} \cup T \cup S$ where $T=\{t_i\mid i \in {\mathbb Z}\}$
and $S=\{s_j\mid 3\leq j \leq r\}$;
\item Relations: the usual braid relations on $S$, along with\\
$\begin{array}{ll}
(Q_1) & t_i t_{i-1} = t_j t_{j-1}
    \quad \mbox{for all $i, j \in {\mathbb Z}$},\\
(Q_2) & s_3 t_i s_3 =t_i s_3 t_i \quad\mbox{for all $i \in {\mathbb Z}$},\\
(Q_3) & s_j t_i  =t_i s_j  \quad \mbox{for all $i\in{\mathbb Z}$ and $4\leq j \leq r$},\\
(R_1) & z t_1 t_0 = t_1 t_0 z, \\
(R_2) & z \langle t_1 t_0\rangle^e = t_0 z \langle t_1 t_0\rangle^{e-1}, \\
(Q_5) & z s_j =s_j z \quad \mbox{for $3 \leq j \leq r$}.
\end{array}$
\end{itemize}

The following claim completes the proof.

\begin{claim}{Claim}
Assuming $(Q_1)$, the relation $(Q_4)$ is equivalent to $(R_1)+(R_2)$.
\end{claim}

\begin{proof}[Proof of Claim]
Suppose that $(Q_4)$ holds, i.e., $zt_i=t_{i-e}z$ for all $i\in{\mathbb Z}$.
Then $(R_1)$ holds because
$$
zt_1t_0
\stackrel{Q_4}{=} t_{1-e}zt_0
\stackrel{Q_4}{=} t_{1-e}t_{-e}z
\stackrel{Q_1}{=} t_{1}t_{0}z.
$$
When $e=2m$,
\begin{align*}
z\langle  t_1t_0\rangle^{e}
&= z\langle  t_1t_0\rangle^{2m} = z(t_1t_0)^m
 \stackrel{Q_1}{=}  z(t_{2m}t_{2m-1})\cdots (t_2t_1)\\
&\stackrel{Q_4}{=} t_0z(t_{2m-1}t_{2m-2})\cdots (t_3t_2)t_1
 \stackrel{Q_1}{=}  t_0z(t_1t_0)^{m-1}t_1
 = t_0 z\langle t_1t_0\rangle^{e-1}.
\end{align*}
When $e=2m+1$,
\begin{align*}
z\langle  t_1t_0\rangle^{e}
&= z\langle t_1t_0\rangle^{2m+1} = z(t_1t_0)^mt_1
 \stackrel{Q_1}{=}  z(t_{2m+1}t_{2m})\cdots (t_3t_2)t_1\\
&\stackrel{Q_4}{=} t_0z(t_{2m}t_{2m-1})\cdots (t_2t_1)
 \stackrel{Q_1}{=}  t_0z(t_1t_0)^{m}
 = t_0 z\langle t_1t_0\rangle^{e-1}.
\end{align*}
Therefore $(R_2)$ holds.

\medskip
Conversely, suppose that both $(R_1)$ and $(R_2)$ hold.
First, we will show that $(Q_4)$ holds for $i=e$, that is, $zt_e=t_0z$.
When $e=2m$,
\begin{align*}
z\langle t_1t_0\rangle^e &=z(t_1t_0)^m \stackrel{Q_1}{=} z (t_{2m}t_{2m-1})\cdots(t_2t_1)
\quad \text{and}\\
t_0 z\langle t_1t_0\rangle^{e-1} & = t_0 z (t_1t_0)^{m-1}t_1
    \stackrel{Q_1}{=} t_0 z(t_{2m-1}t_{2m-2})\cdots(t_3t_2)t_1.
\end{align*}
Since $z\langle  t_1t_0\rangle^{e}=t_0 z\langle t_1t_0\rangle^{e-1}$ by $(R_2)$, we have
$$
z t_{2m} (t_{2m-1}\cdots t_1) = t_0 z (t_{2m-1}\cdots t_1).
$$
Hence $zt_{2m}=t_0z$, that is, $zt_e=t_0z$.
Therefore $(Q_4)$ holds for $i=e$ when $e$ is even.

When $e=2m+1$,
\begin{align*}
z\langle t_1t_0\rangle^e &=z(t_1t_0)^m t_1 \stackrel{Q_1}{=} z (t_{2m+1}t_{2m})\cdots(t_3t_2)t_1
\quad \text{and} \\
t_0 z\langle t_1t_0\rangle^{e-1} &= t_0 z (t_1t_0)^{m}
    \stackrel{Q_1}{=} t_0 z(t_{2m}t_{2m-1})\cdots(t_2t_1).
\end{align*}
Since $z\langle  t_1t_0\rangle^{e}=t_0 z\langle t_1t_0\rangle^{e-1}$ by $(R_2)$, we have
$$
z t_{2m+1} (t_{2m}\cdots t_1) = t_0 z (t_{2m}\cdots t_1).
$$
Hence $zt_{2m+1}=t_0z$, that is, $zt_e=t_0z$.
Therefore $(Q_4)$ holds for $i=e$ when $e$ is odd.

\medskip
Now we will show that if $(Q_4)$ holds for $i=k$, which we denote by  $(Q_{4,k})$,
then $(Q_4)$ holds for $i=k-1$ and $i=k+1$.
Assume that $(Q_4)$ holds for $i=k$.
Because
$$
zt_{k+1}t_k
\stackrel{Q_1}{=} zt_1t_0
\stackrel{R_1}{=} t_1t_0z
\stackrel{Q_1}{=} t_{k+1-e}t_{k-e}z
\stackrel{Q_{4,k}}{=} t_{k+1-e}zt_k,
$$
we have $zt_{k+1}=t_{k+1-e}z$, hence $(Q_4)$ holds for $i=k+1$.
Similarly, because
$$
t_{k-e}t_{k-e-1}z
\stackrel{Q_1}{=} t_1t_0z
\stackrel{R_1}{=} zt_1t_0
\stackrel{Q_1}{=} zt_kt_{k-1}
\stackrel{Q_{4,k}}{=} t_{k-e}zt_{k-1},
$$
we have $t_{k-e-1}z=zt_{k-1}$, hence $(Q_4)$ holds for $i=k-1$.
By induction on $i$, we conclude that $(Q_4)$ holds for all $i\in{\mathbb Z}$.
\end{proof}

\subsection{Garside structures on $B(\infty,\infty,r)$ and $B(de,e,r)$}
\label{sec:Garside_structure}

In this subsection, we show that the new presentations of
$B(\infty,\infty,r)$ and $B(de,e,r)$ give rise to
quasi-Garside structures.

Garside structures were defined by Dehornoy and Paris~\cite{DP99},
in which the strategy and results of
Garside~\cite{Gar69}, Deligne~\cite{Del72}, Brieskorn and Saito~\cite{BS72} still hold.
A Garside structure provides tools for calculating in the group,
for solving word and conjugacy problems, as well as
for giving certain information about
the group (such as being torsion-free).
For a detailed description, see~\cite{DP99,DDGKM14}.
We use the definition in~\cite{Dig06}.

\begin{definition}
A monoid $M$ is said to
be \emph{quasi-Garside} if the following conditions are satisfied:
\begin{enumerate}
\item
$M$ is {\em atomic}---that is,  for every $m \in M$, the number of
factors in a product equal to $m$ is bounded;

\item
$M$ is left- and right-cancellative;

\item
$M$ is a lattice with respect to each of the orders defined by
left divisibility and by right divisibility;

\item
$M$ has a  \emph{Garside element} $\Delta$ for which
the set of left divisors equals the set of right divisors, and this
set generates $M$.
\end{enumerate}
\end{definition}

A quasi-Garside monoid satisfies Ore's conditions~\cite{CP61},
and thus embeds in its group of fractions.

\begin{definition}
Let $M$ be a quasi-Garside monoid with Garside element $\Delta$, and
let $G$ be the group of fractions of $M$.
We identify the elements of $M$ and their images in $G$.
The pair $(M, \Delta)$ is called a {\em quasi-Garside structure} on $G$,
and the triple $(G, M, \Delta)$ or just simply $G$ is called
a {\em quasi-Garside group}.
The quasi-Garside monoid $M$ of $G$ is often denoted by $G^+$.
\end{definition}

When the set of left divisors of the Garside element $\Delta$ is finite,
the word `quasi' may be dropped for quasi-Garside.

If the monoid $M$ is defined by a (positive) presentation with {\em homogeneous}
relations---that is, for every relation, the left and right hand sides have equal length
in the generators---then the first condition
of being atomic is immediately satisfied, with the bound for the number of factors
in a product equal to $m$ being precisely the number of generators in an expression for $m$
(since this number is the same for all expressions for $m$).
This is always the case in the presentations we consider here.

For the second and third conditions of being cancellative and being a lattice,
we introduce the notions of complementedness and completeness of Dehornoy~\cite{Deh03}.
Let $M$ be a monoid defined by a positive presentation $\langle \, S \mid R \,\rangle$.
Let $S^*$ denote the free monoid generated by $S$,
and let $\varepsilon$ denote the empty word.

\begin{definition}
For words $w$, $w'$ on $S\cup S^{-1}$, we say that $w$ {\em right-reverses} to
$w'$, denoted $w \curvearrowright_r w'$,
if $w'$ is obtained from $w$ (iteratively)
\begin{itemize}
\item
either by deleting some subword $u^{-1} u$ for $u\in S^*\setminus\{\varepsilon\}$,

\item
or by replacing some subword $u^{-1}v$ for $u, v\in S^*\setminus\{\varepsilon\}$
with a word $v'u'^{-1}$ such that $uv'=vu'$ is a relation of $R$.
\end{itemize}
\end{definition}

For any $u, v\in S^*$, $u^{-1}v \curvearrowright_r \varepsilon$ implies
that $u=v$ in $M$.

\begin{definition}
The presentation $\langle \, S \mid R \,\rangle $ of $M$ is said to be
\begin{enumerate}
\item
{\em right-complemented} if for any $x, y\in S$, $R$ has at most one relation
of the form $x\cdots = y\cdots$ and no relation of the form $x\cdots =x\cdots$;

\item
{\em right-complete} if for any $u, v\in S^*$, $u=v$ in $M$ implies
$u^{-1}v \curvearrowright_r \varepsilon$.
\end{enumerate}
\end{definition}

The left versions of the above notions are defined symmetrically.
In this subsection, several notions have a left and a right version.
Without `left' or `right', we assume both versions.
For instance, ``$M$ is cancellative'' means ``$M$ is left- and right-cancellative''.

\medskip

Let $B^+(e,e,r)$, $B^+(\infty,\infty,r)$, $B^+(de,e,r)$ and
$B^+({\mathbf{B}}_{r})$ be the monoids
defined by the presentations in Theorems~\ref{beerpresthm}, \ref{binftyrpresthm}
and~\ref{new-pre-deer}
and the presentation in~\eqref{pres:B_r} on page~\pageref{fig:typeAB}, respectively.
It is known that $B^+(e,e,r)$ and $B^+({\mathbf{B}}_{r})$ are Garside.
In the remaining of this subsection,
we will show that $B^+(\infty,\infty,r)$ and $B^+(de,e,r)$ are quasi-Garside.

\begin{remark}
The Brou\'e-Malle-Rouquier presentation for $B(de,e,r)$ does not
give rise to a quasi-Garside structure for all $e\ge 1$.
(Notice that when $e=1$, the group itself is a Garside group
because $B(d,1,r)\cong B({\mathbf{B}}_r)$.)
Assume that the monoid defined by the presentation is quasi-Garside.
Both $s_3t_1s_3=t_1s_3t_1$ and $s_3t_1t_0s_3t_1t_0=t_1t_0s_3t_1t_0s_3$
are common right multiples of $s_3$ and $t_1$, hence
$$
(s_3t_1s_3)\wedge (s_3t_1t_0s_3t_1t_0)
= (s_3t_1) (s_3\wedge t_0s_3t_1t_0)
= s_3t_1,
$$
is also a common right multiple of $s_3$ and $t_1$,
where $\wedge$ denotes the left gcd.
However $s_3t_1$ is not a right multiple of $t_1$, which is a contradiction.
The same argument shows that Shi's presentation for $B(\infty,\infty,r)$
does not give rise to a quasi-Garside structure.
\end{remark}

\subsubsection{Garside structure on $B(\infty,\infty,r)$}
\label{ssec:Garside_Binfty}

Since the presentation of $B^+(\infty,\infty,r)$ is homogeneous,
the monoid is atomic.
The conditions of being cancellative and a lattice
can be checked by using complementedness and completeness.

\begin{lemma}\label{lem:Deh03}
\cite[Corollary 6.2 and Propositions 3.3, 6.7 and 6.10]{Deh03}
Let $M$ be a monoid defined by a complemented and complete presentation
with $S$ the set of generators.
Then the following hold.
\begin{enumerate}
\item The monoid $M$ is cancellative.
\item
Suppose that there exists $S'$ such that
$S\subseteq S'\subseteq S^*$ and for any $u,v\in S'$ there exist
$u',v'\in S'$ with $uv'= vu'$  in $M$.
Then $M$ admits right lcm's.

\item Suppose that there exists $S''$ such that
$S\subseteq S''\subseteq S^*$ and for any $u,v\in S''$ there exist
$u'', v''\in S''$ with $v''u= u''v$ in $M$.
Then $M$ admits left lcm's.
\end{enumerate}
\end{lemma}

The presentation of $B^+(\infty,\infty,r)$ is complemented.
As for completeness, the cases to be considered are identical with
those for $B^+(e,e,r)$ (Figures 7 and 8 in~\cite{CP11}),
hence it can be checked in a manner entirely analogous
to that for $B^+(e,e,r)$ given in \cite{CP11}.
In that paper, the presentation of $B^+(e,e,r)$ was shown to be complete
by using the cube condition on all triples of generators.
Therefore we have the following corollary by Lemma~\ref{lem:Deh03}.

\begin{corollary}\label{cor:cancel}
The monoid $B^+(\infty,\infty,r)$ is  cancellative.
\end{corollary}

Next, we will find sets $S'$ and $S''$ satisfying the conditions
in Lemma~\ref{lem:Deh03}.
Define a map
 $$\begin{array}{rcl}
\psi\colon B^+({\mathbf{B}}_{r-1}) & \to & B^+(\infty,\infty,r) \\[1ex]
b_1 &\mapsto & t_1 t_0 \\[1ex]
b_i &\mapsto & s_{i+1} \quad \mbox{for \ $2\le i\le r-1$}.
\end{array}$$
It is easy to see that $\psi$ is a well-defined monoid homomorphism
because the defining relations in $B^+({\mathbf{B}}_{r-1})$ can be realized by
the relations in $B^+(\infty,\infty,r)$. For example,
$t_1t_0s_3t_1t_0s_3=s_3t_1t_0s_3t_1t_0$ holds in
$B^+(\infty,\infty,r)$ (see the proof of
Theorem~\ref{binftyrpresthm}), hence
$\psi(b_1)\psi(b_2)\psi(b_1)\psi(b_2)=\psi(b_2)\psi(b_1)\psi(b_2)\psi(b_1)$
holds in $B^+(\infty,\infty,r)$.
We remark that $\psi$ is injective.
To see this, consider the composition with the morphism
from $B(\infty,\infty,r)$ to $B({\mathbf{A}}_{r-1})$ which maps $t_i$ to $\sigma_1$
for $i\in{\mathbb Z}$ and $s_j$ to $\sigma_{j-1}$ for $3\le j\le r$.
The composition is the well-known embedding of $B^+({\mathbf{B}}_{r-1})$ into $B^+({\mathbf{A}}_{r-1})$
which maps $b_1$ to $\sigma_1^2$ and $b_i$ to $\sigma_i$ for $2\le i\le r-1$.

\medskip
The classical Garside element of the braid group $B({\mathbf{B}}_{r-1})$, denoted by
$\Delta_{{\mathbf{B}}_{r-1}}$, is the lcm of the generators
$\{b_1, b_2, \ldots, b_{r-1}\}$. It is a central element  of $B^+({\mathbf{B}}_{r-1})$,
written as
$$
\Delta_{{\mathbf{B}}_{r-1}}=(b_{r-1}b_{r-2}\cdots b_1)^{r-1}.
$$

Let $\Lambda\in B^+(\infty,\infty,r)$ be the image of $\Delta_{{\mathbf{B}}_{r-1}}$
under $\psi$.
Then $\Lambda$ has the factorization
\begin{equation*}
\Lambda=\psi(\Delta_{{\mathbf{B}}_{r-1}})=(At_1t_0)^{r-1},
\end{equation*}
where $A=s_rs_{r-1}\cdots s_3$.

Let $L(\Lambda)$ and $R(\Lambda)$ denote the sets of all left and right divisors of
$\Lambda$, respectively.

\begin{proposition}\label{GarsideElt}
The element $\Lambda$ is a Garside element of $B^+(\infty,\infty,r)$. That is,
$L(\Lambda)=R(\Lambda)$, and $L(\Lambda)$ generates $B^+(\infty,\infty,r)$.
\end{proposition}

We give the proof of the above proposition in \S\ref{ssec:GarsideElt}.

The set $L(\Lambda)=R(\Lambda)$ meets the needs of $S'$ and $S''$ in Lemma~\ref{lem:Deh03}.
Therefore $B^+(\infty,\infty,r)$ admits lcm's.
It is easy to see that if a cancellative monoid admits
lcm's then it admits gcd's.
(For example, see Lemma 2.23 in~\cite{DDGKM14}).

So far, we have shown that the monoid $B^+(\infty,\infty,r)$ satisfies all the
conditions in the definition of quasi-Garside monoids.

\begin{theorem}\label{thm:Garside-infty}
The presentation for $B(\infty,\infty,r)$ in Theorem~\ref{binftyrpresthm}
gives rise to a quasi-Garside structure,
where $B^+(\infty,\infty,r)$ is the quasi-Garside monoid and $\Lambda$ is a Garside element.
\end{theorem}

We will see in \S\ref{ssec:Embed_Bdeer_Br}
that the braid group  $B(\infty,\infty,r)$ is isomorphic to
the affine braid group  $B({\widetilde{\mathbf A}_{r-1}}
)$.
In~\cite{Dig06}, Digne proposed a dual presentation of $B({\widetilde{\mathbf A}_{r-1}}
)$
which gives a quasi-Garside structure.
This is different from the quasi-Garside structure on $B({\widetilde{\mathbf A}_{r-1}}
)\cong B(\infty,\infty,r)$
given in Theorem~\ref{thm:Garside-infty}.

\subsubsection{Garside structure on $B(de,e,r)$}

Recall from Corollary~\ref{bdeerthm} that $B(de,e,r) \cong C^e_\infty \ltimes B(\infty,\infty,r)$
where $C^e_\infty=\langle  z \rangle $ is an infinite cycle group.
From the presentation for $B(de, e,r)$ in Theorem~\ref{new-pre-deer},
$C^e_\infty$ acts on $B^+(\infty,\infty,r)$ by $zt_iz^{-1}=t_{i-e}$ and
$zs_jz^{-1}=s_j$ for $i\in{\mathbb Z}$ and $3\le j\le r$.

From \S\ref{ssec:Garside_Binfty}, $(B(\infty,\infty,r), B^+(\infty,\infty,r), \Lambda)$
is a quasi-Garside group.
On the other hand, $(C^e_\infty, (C^e_\infty)^+, z)$ is a Garside group, where
$(C^e_\infty)^+=\{\, z^n \mid n\ge 0\, \}$.

Picantin~\cite{Pic01} showed that the crossed product of Garside monoids
is a Garside monoid.
For semidirect products which are a special case of crossed products,
it was directly proved in~\cite{Lee07}.

\begin{lemma}[{\cite[Theorem 4.1]{Lee07}}]\label{thm:Lee07}
Let $(G, G^+, \Delta_G)$ and $(H, H^+, \Delta_H)$ be quasi-Garside groups.
If $\rho$ is an action of $G$ on $H^+$  and $\Delta_H$ is fixed under $\rho$,
then $(G \ltimes_\rho H, G^+\ltimes_\rho H^+, (\Delta_G, \Delta_H))$ is a quasi-Garside group.
\end{lemma}

Theorem 4.1 in \cite{Lee07} is indeed stated not for quasi-Garisde groups
but for Garside groups,
but its proof does not use finiteness of divisors of
$\Delta_G$ or $\Delta_H$.
Hence the above lemma is true.

Since $z\Lambda=\Lambda z$, the Garside element $\Lambda$ of $B^+(\infty,\infty,r)$
is fixed under the action of $C^e_\infty$.
Applying Lemma~\ref{thm:Lee07}, we have the following result.

\begin{theorem}
The presentation for $B(de,e,r)$ in Theorem~\ref{new-pre-deer} gives rise to a quasi-Garside structure
on the imprimitive braid group  $B(de,e,r) \cong C^e_\infty \ltimes B(\infty,\infty,r)$, where
$B^+(de,e,r) \cong (C^e_\infty)^+ \ltimes B^+(\infty,\infty,r)$ is
the quasi-Garside monoid and $z\Lambda$ is a Garside element.
\end{theorem}

Indeed, $z^p \Lambda^q$ is a Garside element for any positive exponents $p$ and $q$.
A better choice would be $z^{\frac{r}{e\wedge r}}\Lambda^{\frac{e}{e\wedge r}}$
because it is the generator of the center of $B(de,e,r)$.
Similarly $z^r\Lambda^e$ is also a good choice.
However, if $p/q\ne r/e$, then $z^p\Lambda^q$ does not have a central power.

\subsubsection{Proof of Proposition~\ref{GarsideElt}}
\label{ssec:GarsideElt}
When $r=2$, the assertion is obvious.
Hence we assume $r\ge 3$.

The  properties required for the element $\Lambda$
to be a Garside element of $B^+(\infty, \infty, r)$ are largely inherited
from the Garside element $\Delta_{{\mathbf{B}}_{r-1}}$
of $B^+({\mathbf{B}}_{r-1})$.

Let ``$\doteq$'' denote the equivalence in each of the positive monoids
$B^+({\mathbf{B}}_{r-1})$ and $B^+(\infty,\infty,r)$.

Notice that $\tau:B(\infty,\infty,r)\to B(\infty,\infty,r)$, defined by
$\tau(t_i)=t_{i+1}$ and $\tau(s_j)=s_j$ for $i\in{\mathbb Z}$ and $3\le j\le r$,
induces an automorphism of the monoid $B^+(\infty,\infty,r)$.
Then for all $k\in{\mathbb Z}$
$$\tau^k(\Lambda)\doteq \Lambda$$
because $\tau^k(A)=A$ and $\tau^k(t_1t_0)=t_{k+1}t_k\doteq t_1t_0$.

Since $\Delta_{{\mathbf{B}}_{r-1}}$ is central in $B^+({\mathbf{B}}_{r-1})$,
$\Delta_{{\mathbf{B}}_{r-1}} b_i \doteq  b_i\Delta_{{\mathbf{B}}_{r-1}}$  for all $1\le i\le r-1$.
Therefore
\begin{eqnarray*}
&& \Lambda t_1t_0 \doteq t_1t_0\Lambda,\\
&& \Lambda s_j \doteq s_j \Lambda \qquad\mbox{for } 3\le j\le r.
\end{eqnarray*}

\begin{lemma}
$\Lambda\doteq (At_r)(At_{r-1})\cdots (At_1)$.
\end{lemma}

\begin{proof}
When $r=3$, the assertion is true because
$$
\Lambda
= (s_3t_1t_0)^2 = s_3t_1t_0s_3t_1t_0
\doteq s_3t_3t_2s_3t_2t_1\\
\doteq s_3t_3s_3t_2s_3t_1= (At_3)(At_2)(At_1).
$$
Hence we assume $r\ge 4$.

Notice that the following identities hold:
\begin{eqnarray*}
&& t_i(At_i)\doteq (At_i)s_3  \qquad\mbox{for $i\in{\mathbb Z}$},\\
&& s_j(At_i)\doteq(At_i)s_{j+1}  \qquad \mbox{for $i\in{\mathbb Z}$ and $3\le j\le r-1$.}
\end{eqnarray*}
One can prove it directly using the defining relations, or using the fact that
$\{ t_i, s_3,\ldots, s_r\}$ for any $i\in{\mathbb Z}$
satisfies the braid relations described by the
following diagram.
$$
\begin{xy}
(10,0) *++={\rule{0pt}{4pt}} *\frm{o};
(20,0) *++={\rule{0pt}{4pt}} *\frm{o}**@{-};
(30,0) *++={\rule{0pt}{4pt}} *\frm{o}**@{-};
(40, 0) *++={\dots}  **@{-};
(50, 0) *++={\rule{0pt}{4pt}} *\frm{o} **@{-};
(13,-3) *++={t_i};
(23,-3) *++={s_3};
(33,-3) *++={s_4};
(53,-3) *++={s_r}
\end{xy}
$$
By moving $t_k$ from the left to the right using the above identities, we have
\begin{equation}\label{eq:move_t_k}
t_k (At_k)(At_{k-1})\cdots (At_2)\doteq (At_k)(At_{k-1})\cdots (At_2) s_{k+1}
\qquad\text{for } 2\le k\le r-1.
\end{equation}
Now, we claim that the following identity holds.
\begin{equation}\label{eqn:pow}
(At_1t_0)^k\doteq (At_{k+1})(At_k)\cdots (At_2)\, (s_{k+1}\cdots s_3)\, t_1
\qquad\text{for } 1\le k\le r-1.
\end{equation}
The above equality is obvious for $k=1$.
Using induction on $k$, assume that \eqref{eqn:pow} is true for some $k$
with $1\le k\le r-2$.
Then by~\eqref{eq:move_t_k}
\begin{align*}
(At_1t_0)^{k+1}
&= (At_1t_0) \, (At_1t_0)^k\\
&\doteq A t_{k+2}t_{k+1} \, (At_{k+1}) (At_k)\cdots (At_2)\, (s_{k+1}\cdots s_3)\, t_1\\
&\doteq At_{k+2}\,  (At_{k+1})(At_k)\cdots (At_2)\,  s_{k+2}\, (s_{k+1}\cdots s_3)\,  t_1.
\end{align*}
This shows that \eqref{eqn:pow} is true.

Putting $k=r-1$ to \eqref{eqn:pow}, we obtain
$\Lambda\doteq (At_r)(At_{r-1})\cdots (At_1)$.
\end{proof}

\begin{lemma}\label{lem:Lambda_balanced}
$\Lambda g\doteq \tau^r(g)\Lambda$ for all $g\in B^+(\infty,\infty,r)$.
\end{lemma}

\begin{proof}
Note that $\tau^r(t_i)=t_{i+r}$ for $i\in{\mathbb Z}$ and $\tau^r(s_j)=s_j$ for $3\le j\le r$.
Since $\Lambda s_j \doteq s_j \Lambda$ for $3\le j\le r$,
it suffices to show that $\Lambda t_i\doteq t_{i+r}\Lambda$ for all $i\in{\mathbb Z}$.

Since $\Lambda A\doteq A\Lambda$ and $\tau^k(\Lambda)\doteq \Lambda$ for $k\in{\mathbb Z}$,
$$
(At_r)\, \Lambda
\doteq (At_r)\, \tau^{-1}(\Lambda)
\doteq (At_r)\, (At_{r-1})\cdots (At_1) (At_0)
\doteq \Lambda (At_0)\doteq A\Lambda t_0.
$$
Hence $At_r\Lambda\doteq A\Lambda t_0$.
Because $B^+(\infty,\infty,r)$ is cancellative,
$$
t_r\Lambda\doteq \Lambda t_0.
$$
Applying $\tau^i$ to the above identity, we have
$t_{i+r}\Lambda\doteq \Lambda t_i$.
\end{proof}

In \cite[Proposition 3.5]{CP11}, the following was shown.

\begin{lemma}\label{lem:CP}
Let $M$ be a cancellative monoid and let $h\in M$.
If there is an automorphism $\phi$ of $M$ such that
$hg=\phi(g)h$ for all $g\in M$,
then the set of left divisors of $h$ is the same as the set of right divisors of $h$.
\end{lemma}

Now, we are ready to show that $\Lambda$ is a Garside element.

\begin{proof}[End of proof of Proposition~\ref{GarsideElt}]
Since each $b_i$ ($1\le i\le r-1$)
is a left divisor of $\Delta_{{\mathbf{B}}_{r-1}}$ in $B^+({\mathbf{B}}_{r-1})$,
$s_3,\ldots,s_r$ and $t_1t_0$ are left divisors of $\Lambda$ in $B^+(\infty, \infty, r)$.
Since each $t_i$ ($i \in{\mathbb Z}$) is a left divisor of $t_1t_0$, the set of left
divisors of $\Lambda$ contains $\{t_i\mid i\in{\mathbb Z}\}\cup \{s_3,\ldots,s_r\}$,
hence it generates $B^+(\infty,\infty,r)$.
By Corollary~\ref{cor:cancel} and Lemmas~\ref{lem:Lambda_balanced} and \ref{lem:CP},
the set of left divisors of $\Lambda$ equals the set of right
divisors of $\Lambda$.
Consequently, $\Lambda$ is a Garside element.
\end{proof}

\section{Geometric interpretation and Applications}
\label{sec:GeomBraid}

In this section, we explore some properties of $B(de,e,r)$
by using the interpretation of $B(de,e,r)$ as a geometric braid group.

\subsection{Interpretation as geometric braids on $r+1$ strings}
\label{ssec:Embed_Bdeer_Br}
In~\cite{BMR98}, Brou\'e, Malle and Rouquier constructed an isomorphism
from $B(d,1,r)$ to $B({\mathbf{B}}_r)$ and an embedding of $B(de,e,r)$ into $B({\mathbf{B}}_r)$.
This subsection begins with reviewing them in our setting.

Consider the braid group $B(d,1,r)$.
By putting $e=1$ to the relation $z t_i = t_{i-e}z$ of $B(de,e,r)$
in Theorem~\ref{new-pre-deer},
we have $z t_i = t_{i-1} z$ for $i\in{\mathbb Z}$, hence for every $i\in{\mathbb Z}$
$$
t_i=z^{-i}t_0z^i.
$$
Then it is straightforward to simplify the presentation of $B(d,1,r)$ in Theorem~\ref{new-pre-deer}
to the presentation illustrated by the diagram in Figure~\ref{fig:typeB}.
Notice that it is the same as the diagram for $B({\mathbf{B}}_r)$ in Figure~\ref{fig:typeAB}(b),
where $b_1$ and $b_2$ are replaced with $z$ and $t_0$, respectively,
and $b_i$ is replaced with $s_i$ for each $3\le i\le r$.
Therefore the groups $B(d,1,r)$ and  $B({\mathbf{B}}_r)$ are isomorphic
by the following map.
$$\begin{array}{rcll}
B(d,1,r) & \overset{\cong}{\longrightarrow} & B({\mathbf{B}}_r) & \\[1ex]
z & \mapsto & b_1 & \\[1ex]
t_i & \mapsto & b_1^{-i} b_2 b_1^i & (i\in{\mathbb Z}) \\[1ex]
s_j & \mapsto & b_j & (3\leq j \leq r)
\end{array}$$

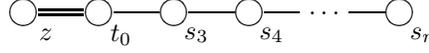
\begin{figure}
$$\begin{xy}
(10,0) *++={\rule{0pt}{4pt}} *\frm{o};
(20,0) *++={\rule{0pt}{4pt}} *\frm{o}**@{=};
(30, 0) *++={\rule{0pt}{4pt}} *\frm{o} **@{-};
(40, 0) *++={\rule{0pt}{4pt}} *\frm{o} **@{-};
(50, 0) *++={\dots}  **@{-};
(60, 0) *++={\rule{0pt}{4pt}} *\frm{o} **@{-};
(13,-3) *++={z};
(23,-3) *++={t_0};
(33,-3) *++={s_3};
(43,-3) *++={s_4};
(63,-3) *++={s_r}
\end{xy}
$$
\caption{Diagram for $B(d,1,r)$}
\label{fig:typeB}
\end{figure}

In \S\ref{ssec:Maps_Bdeer},
we saw that $B(\infty,\infty,r)$ is a subgroup of $B(d,1,r)$, generated by $T \cup S$
where $T= \{t_i \mid i \in {\mathbb Z}\}$ and  $S = \{s_j \mid 3 \leq j \leq r\}$.
Similarly, using the map $\iota_{1}^{e}$ which  embeds $B(de,e,r)$
into $B(d,1,r)$ as a subgroup of index $e$
by sending $z$ (of $B(de,e,r)$) to $z^e$ (of $B(d,1,r)$),
we have that
$B(de,e,r)$ is isomorphic to a subgroup of $B(d,1,r)$ generated by $\{z^e\} \cup T \cup S$.
These embeddings can be summarized as follows.

\begin{equation*}
\renewcommand{\arraystretch}{1.3}
\begin{array}{rcccccll}
B(\infty,\infty,r) & \hookrightarrow & B(de,e,r) & \hookrightarrow &
B(\mathbf{B}_r) & \hookrightarrow & B_{r+1} & \\
 &  & z & \mapsto & b_1^e & \mapsto & \sigma_1^{2e} & \\
t_i  & \mapsto & t_i & \mapsto & b_1^{-i}b_2b_1^i & \mapsto
  & \sigma_1^{-2i}\sigma_2\sigma_1^{2i} &  (i\in{\mathbb Z})\\
s_j & \mapsto & s_j & \mapsto & b_j & \mapsto & \sigma_j  & (3\le j\le r)
\end{array}
\end{equation*}

\medskip

\begin{proposition}
\label{bbr}
\begin{enumerate}
\item The braid group $B(de,e,r)$ is isomorphic to the subgroup
of\/ $B({\mathbf{B}}_r)$ of index $e$ generated by
$\{b_1^e\} \cup
\{b_1^{-i} b_2 b_1^i \mid i \in {\mathbb Z}\} \cup
\{b_j \mid 3 \leq j \leq r\}$.
\item The braid group $B(\infty,\infty,r)$ is isomorphic to
the subgroup of $B({\mathbf{B}}_r)$ generated by
$\{b_1^{-i} b_2 b_1^i \mid i \in {\mathbb Z}\} \cup \{b_j \mid 3 \leq j \leq r\}$.
\end{enumerate}
\end{proposition}

The first statement of the above proposition is Proposition~3.8 in~\cite{BMR98}.
The following is a direct consequence of the fact that
$B(de,e,r)$ is a finite index subgroup of $B({\mathbf{B}}_r)$ and that $B({\mathbf{B}}_r)$ is a Garside group.

\begin{corollary}
$B(de,e,r)$ has a finite $K(\pi,1)$,
and is biautomatic.
In particular, the word and conjugacy problems in $B(de,e,r)$ are solvable.
\end{corollary}

\begin{proof}
Every Garside group is biautomatic~\cite{DP99} and
has a finite $K(\pi, 1)$~\cite{CMW04}.
Any finite index subgroup of a biautomatic group is biautomatic, and
the word and conjugacy problems are solvable in biautomatic groups~\cite{WordProc}.
It is obvious that any finite index subgroup of a group with a finite $K(\pi, 1)$
has a finite $K(\pi, 1)$.
Since $B({\mathbf{B}}_r)$ is a Garside group and
$B(de,e,r)$ is a subgroup of $B({\mathbf{B}}_r)$ of index $e$, we are done.
\end{proof}

Recall that $B(d,1,r)\cong B({\mathbf{B}}_r)\cong B_{r+1,1}$
and that $\operatorname{wd}:B_{r+1,1}\to{\mathbb Z}$ is defined
by $\operatorname{wd}(\sigma_1^2)=1$ and $\operatorname{wd}(\sigma_j)=0$ for $2\le j\le r$.
Under the identification $B(d,1,r) \cong B_{r+1,1}$,
the homomorphism $B(d,1,r) \to C^1_\infty\cong{\mathbb Z}$ in the the exact sequence in \S\ref{ssec:Maps_Bdeer}
is the same as the winding number $\operatorname{wd}:B_{r+1,1}\to{\mathbb Z}$.
Hence we have the following commutative diagram, where the rows are exact.
$$
\xymatrix{
0 \ar[r] & B(\infty,\infty,r) \ar[r]
    & B(d,1,r) \ar[r] \ar[d]^{\cong}
    & C^1_\infty \ar[r] \ar[d]^{\cong} & 0 \\
0 \ar[r] & \ker(\operatorname{wd}) \ar[r]
    & B_{r+1,1} \ar[r]^{\operatorname{wd}}
    & {\mathbb Z} \ar[r] & 0 \\
}$$
Then it follows that $B(\infty,\infty,r)$ is isomorphic to
$\ker(\operatorname{wd})= \{\, g\in B_{r+1,1}\mid \operatorname{wd}(g)=0 \,\}$.

Let $\operatorname{wd}_e \colon B_{r+1,1}\to{\mathbb Z}/e$ be the homomorphism defined by $\operatorname{wd}_e(g)=\operatorname{wd}(g)\bmod e$.
Then $\ker(\operatorname{wd}_e)=\{\, g\in B_{r+1,1}\mid \operatorname{wd}(g)\equiv 0 \mod{e} \,\}$
is a subgroup of $B_{r+1,1}$ of index $e$.
Because the subgroup generated by $\{\sigma_1^{2e}\} \cup
\{\sigma_1^{-2i} \sigma_2 \sigma_1^{2i} \mid i \in {\mathbb Z}\} \cup
\{\sigma_j \mid 3 \leq j \leq r\}$ is also of index $e$ in $B_{r+1,1}$ by Proposition~\ref{bbr}
and because it is a subgroup of $\ker(\operatorname{wd}_e)$, it must coincide with $\ker(\operatorname{wd}_e)$.
Therefore $B(de,e,r)$ is isomorphic to $\ker(\operatorname{wd}_e)$.

From the above discussions, Proposition~\ref{bbr} can be translated into the context of braids
on $r+1$ strings and winding numbers as follows.

\begin{corollary}
\label{cor:GeoBr}
We have the following isomorphisms.
\begin{align*}
B(d,1,r) &\cong B_{r+1,1},\\
B(de,e,r) &\cong \{\, g\in B_{r+1,1}\mid \operatorname{wd}(g)\equiv 0 \mod{e} \,\},\\
B(\infty,\infty,r) &\cong \{\, g\in B_{r+1,1}\mid \operatorname{wd}(g)=0 \,\}.
\end{align*}
The isomorphism $B(d,1,r) \cong B_{r+1,1}$ is given by
$$
z\mapsto \sigma_1^{2}, \quad
t_i\mapsto \sigma_1^{-2i}\sigma_2\sigma_1^{2i} \quad\text{for } i\in{\mathbb Z},\quad
s_j\mapsto \sigma_j \quad\text{for } 3\le j\le r.
$$
\end{corollary}

In this way, elements of $B(de,e,r)$ and $B(\infty, \infty, r)$
may be {\em visualised} as geometric braids.
Notice that the braid groups $B(\infty,\infty,r)$ and $B(\widetilde{\mathbf{A}}_{r-1})$
are isomorphic to the same subgroup of $B_{r+1,1}$,
hence we have the following.

\begin{corollary}\label{cor:isomorphism}
The braid group $B(\infty,\infty,r)$ is isomorphic to the braid
group $B(\widetilde{\mathbf{A}}_{r-1})$.
\end{corollary}

Similarly to the case of $B(e,e,r)$, define a map
$$\begin{array}{rcl}
\tau \colon B(de,e,r) & \to & B(de,e,r) \\
z &\mapsto & z \\
t_i &\mapsto & t_{i+1} \qquad (i\in{\mathbb Z}) \\
s_j &\mapsto & s_j \qquad (3\le j\le r).
\end{array}$$
Then $\tau$ is an automorphism of $B(de,e,r)$, and
$\tau(g)=\sigma_1^{-2}g\sigma_1^2$ for all $g\in B(de,e,r)$
when $B(de,e,r)$ is viewed as a subgroup of $B_{r+1,1}$.

If $e=1$, then $\tau$ is an inner automorphism, that is,
$\tau(g)=z^{-1}gz$ for all $g\in B(d,1,r)$.
But this is not necessarily the case for $e\ge 2$.

\begin{proposition}
The automorphism $\tau:B(de,e,r)\to B(de,e,r)$ is an inner automorphism if and only if
$r$ and $e$ are relatively prime.
\end{proposition}

\begin{proof}
Identify $B(de,e,r)$ with $\{\, g\in B_{r+1,1}\mid \operatorname{wd}(g)\equiv 0 \mod{e} \,\}$.

Suppose that $\tau$ is an inner automorphism of $B(de,e,r)$.
Then there exists $x\in B_{r+1,1}$ with $\operatorname{wd}(x)\equiv 0 \mod{e}$ such that
$$
x^{-1}gx=\tau(g)=\sigma_1^{-2}g\sigma_1^2
$$
for all $g\in B(de,e,r)$.
Therefore $x\sigma_1^{-2}$ commutes with all the elements in $B(de,e,r)$.
Hence it commutes with $\sigma_1^{2e},\sigma_2,\ldots,\sigma_r$.
It is known that if $h\in B_{r+1}$ commutes with $\sigma_i^k$ for some $k\ne 0$, then
$h$ commutes with $\sigma_i$ (see~\cite{FRZ96}).
Therefore $x\sigma_1^{-2}$ commutes with $\sigma_i$ for all $1\le i\le r$,
hence $x\sigma_1^{-2}$  belongs to the center of $B_{r+1}$, which is
the infinite cyclic group generated by the full twist
$\Delta^2$, where $\Delta=(\sigma_1)(\sigma_2\sigma_1)\cdots(\sigma_r\sigma_{r-1}\cdots\sigma_1)$.
Hence $x\sigma_1^{-2}=\Delta^{2k}$ for some $k\in{\mathbb Z}$.
So $x=\Delta^{2k}\sigma_1^2$.
Then
$$
\operatorname{wd}(x)=\operatorname{wd}(\Delta^{2k})+\operatorname{wd}(\sigma_1^2)=kr+1\equiv 0 \mod{e}.
$$
Therefore $r$ and $e$ are relatively prime.

Conversely, suppose that $r$ and $e$ are relatively prime.
Then $kr+1\equiv 0 \mod{e}$ for some $k\in{\mathbb Z}$.
Let $x=\sigma_1^2\Delta^{2k}\in B_{r+1}$.
Since $x$ is 1-pure and $\operatorname{wd}(x)=kr+1\equiv 0 \mod{e}$, $x\in B(de,e,r)$.
Since $\Delta^{2k}$ is central,
$x^{-1}gx=\sigma_1^{-2}g\sigma_1^2=\tau(g)$ for all $g\in B(de,e,r)$,
hence $\tau$ is an inner automorphism.
\end{proof}

The next proposition will be used in the
study of discreteness of translation numbers (in \S\ref{ssec:tran})
and classification of periodic elements (in \S\ref{sec:PeriodElt}).

\begin{proposition}\label{prop:conj}
The embedding $\iota_1^e:B(de,e,r)\to B(d,1,r)$ induces
a finite-to-one map on the sets of conjugacy classes.
More precisely, for $g, h\in B(de,e,r)$,
$\iota_1^e(g)$ and $\iota_1^e(h)$ are conjugate in $B(d,1,r)$ if and
only if $g$ is conjugate to $\tau^k(h)$ in $B(de,e,r)$ for some $0\le k<e$.
\end{proposition}

\begin{proof}
Using Corollary~\ref{cor:GeoBr}, we identify $B(d,1,r)$ and $B(de,e,r)$ with $B_{r+1,1}$ and
$\{\, g\in B_{r+1,1}\mid \operatorname{wd}(g)\equiv 0 \mod{e} \, \}$, respectively.
Let $g, h\in B(de,e,r)$.

Suppose that $g$ is conjugate to $\tau^k(h)$ in $B(de,e,r)$ for some $0\le k<e$.
Since $\tau^k(h)=\sigma_1^{-2k}h\sigma_1^{2k}$ is conjugate to $h$ in $B(d,1,r)$,
$g$ and $h$ are conjugate in $B(d,1,r)$.

Conversely, suppose that $g$ and $h$ are conjugate in $B(d,1,r)$.
Then $h=x^{-1}gx$ for some $x\in B(d,1,r)=B_{r+1,1}$.
Let $\operatorname{wd}(x)\equiv -k \mod{e}$ for some $0\le k<e$.
Let $y=x\sigma_1^{2k}$.
Then $y\in B(de,e,r)$ as $\operatorname{wd}(y)\equiv 0 \mod{e}$, and
$$
y^{-1}gy=\sigma_1^{-2k}x^{-1}gx\sigma_1^{2k}
=\sigma_1^{-2k}h\sigma_1^{2k}=\tau^k(h).
$$
Therefore $g$ is conjugate to $\tau^k(h)$ in $B(de,e,r)$.
\end{proof}

\subsection{Uniqueness of roots up to conjugacy}

The following are well-known results on the uniqueness of roots in braid groups.

\smallskip

\begin{enumerate}
\item
(Gonz\'alez-Meneses \cite{Gon03})
Let $g$ and $h$ be elements of\/ $B({\mathbf{A}}_r)$
such that $g^k=h^k$ for some nonzero integer $k$.
Then $g$ and $h$ are conjugate in $B({\mathbf{A}}_r)$.

\item
(Bardakov~\cite{Bar}, Kim and Rolfsen~\cite{KR03})
Let $g$ and $h$ be pure braids in $B({\mathbf{A}}_r)$
such that $g^k=h^k$ for some nonzero integer $k$.
Then $g$ and $h$ are equal.

\item
(Lee and Lee~\cite{LL10})
Let\/ $G$ be one of the braid groups of types ${\mathbf{A}}_r$, ${\mathbf{B}}_r$, $\widetilde {\mathbf{A}}_{r-1}$ and 
$\widetilde {\mathbf{C}}_{r-1}$.
If\/ $g,h\in G$ are such that
$g^k=h^k$ for some nonzero integer $k$,
then $g$ and $h$ are conjugate in $G$.
\end{enumerate}

The first of these was conjectured by Makanin~\cite{Mak71}
in the early seventies, and proved by Gonz\'alez-Meneses.
The second was initially proved by Bardakov
by combinatorial arguments, and it follows easily
from the biorderability of pure braids by Kim and Rolfsen.
The third result is a generalization of the other two and comes from
the following theorem, by viewing the braid groups of types ${\mathbf{B}}_r$,
$\widetilde {\mathbf{A}}_{r-1}$ and  $\widetilde{\mathbf{C}}_{r-1}$ as subgroups
of $B_{r+1}$ consisting of partially pure braids.

\begin{theorem}[\cite{LL10}]
\label{thm:UniqueRoot}
Let $P$ be a subset of\/ $\{1,\ldots,r+1\}$ with $1\in P$.
Let $g$ and $h$ be $P$-pure $(r+1)$-braids
such that $g^k=h^k$ for some nonzero integer $k$.
Then there exists a $P$-straight $(r+1)$-braid $x$
with $h=x^{-1} g x$ and $\operatorname{wd}(x)=0$.
\end{theorem}

Applying Theorem~\ref{thm:UniqueRoot} to $1$-pure $(r+1)$-braids,
and using the isomorphisms in Corollary~\ref{cor:GeoBr},
we obtain the following uniqueness of roots up to conjugacy
in $B(de,e,r)$ and $B(\infty,\infty,r)$.

\begin{corollary}
Let $g,h\in B(de,e,r)$ be such that $g^k=h^k$
for some nonzero integer $k$.
Then $g$ and $h$ are conjugate in $B(de,e,r)$.
Furthermore, a conjugating element $x$ can be chosen from the subgroup
$B(\infty,\infty,r)$ so that $h=x^{-1}gx$.
\end{corollary}

\begin{corollary}
If\/ $g,h\in B(\infty, \infty, r)$ are such that
$g^k=h^k$ for some nonzero integer $k$,
then $g$ and $h$ are conjugate in $B(\infty, \infty, r)$.
\end{corollary}

\noindent
\textbf{Question.}\ \
Does the uniqueness of roots up to conjugacy hold in $B(e,e,r)$
and in the braid groups of real reflection groups of types other than
${\mathbf{A}}_r$, ${\mathbf{B}}_r$, $\widetilde{\mathbf{A}}_{r-1}$ and $\widetilde{\mathbf C}_{r-1}$?

\subsection{Discreteness of translation numbers}
\label{ssec:tran}

Translation numbers, introduced by Gersten and Short~\cite{GS91},
are  quite useful since it has both algebraic and geometric aspects.
For a finitely generated group $G$ and a finite set $X$ of semigroup generators for $G$,
the \emph{translation number} of an element $g\in G$ with respect to $X$  is defined by
$$t_{G,X}(g)=\liminf_{n\to \infty}\frac{|g^n|{}_X}n,$$
where $|\cdot|_X$ denotes the minimal word-length in the alphabet $X$.
When $A$ is a set of group generators, $|g|_A$ and $t_{G,A}(g)$ indicate
$|g|_{A\cup A^{-1}}$ and $t_{G, A\cup A^{-1}}(g)$, respectively.
Kapovich~\cite{Kap97} and Conner~\cite{Con00}
suggested the following notions:
a finitely generated group $G$ is said to be
\begin{enumerate}
\item \emph{translation separable}
if for some (and hence for any) finite set $X$ of semigroup generators for $G$
the translation numbers of non-torsion elements are strictly positive;

\item \emph{translation discrete} if it is translation separable
and for some (and hence for any) finite set $X$ of semigroup generators for $G$
the set $t_{G, X}(G)$ has 0 as an isolated point;

\item \emph{strongly translation discrete} if it is translation separable
and for some (and hence for any) finite set $X$ of semigroup generators for $G$
and for any real number $r$ the number of conjugacy classes
$[g]=\{h^{-1}gh:h\in G\}$
with $t_{G, X}(g)\le r$ is finite. (The translation number is constant on
each conjugacy class.)
\end{enumerate}

There are several results on translation numbers in
geometric and combinatorial groups.
Biautomatic groups are translation separable~\cite{GS91}.
Word hyperbolic groups are strongly translation discrete,
and moreover, the translation numbers in a word hyperbolic group are rational
with uniformly bounded denominators~\cite{Gro87,BGSS91,Swe95}.
Artin groups of finite type are translation discrete~\cite{Bes99}.
Garside groups are strongly translation discrete,
and the translation numbers are rational with uniformly
bounded denominators~\cite{CMW04,Lee07, LL07}.

Translation numbers of the braid groups
$B(\infty,\infty,r)$ and $B(de,e,r)$ have the
following properties.

\begin{theorem}
The braid group $B(\infty,\infty,r)$ is translation discrete, and
the braid group $B(de,e,r)$ is strongly translation discrete.
\end{theorem}

\begin{proof}
It is known that a subgroup of a translation discrete group
is translation discrete~\cite{Con98}.
Since $B(d,1,r)\cong B({\mathbf{B}}_r)$ is strongly translation discrete~\cite{Lee07}
and since $B(de,e,r)$ and $B(\infty,\infty,r)$ are subgroups of $B(d,1,r)$,
the groups $B(de,e,r)$ and $B(\infty,\infty,r)$ are translation discrete.

Let $G=B(d,1,r)$ and $H=B(de,e,r)$.
Identify $G$ and $H$ with $B_{r+1,1}$ and
$\{\, g\in B_{r+1,1}\mid \operatorname{wd}(g)\equiv 0 \mod{e} \,\}$, respectively.
Choose a finite set of generators, say $X$, for $H$.
Then $Y=X\cup\{\sigma_1^2\}$ is a finite set of generators for $G$.
Choose any real number $r$, and let
\begin{align*}
\mathcal A&=\{h\in H\mid t_{H,X}(h)\le r\},\\
\mathcal B&=\{g\in G \mid t_{G,Y}(g)\le r\}.
\end{align*}
Notice that for any $h\in H$, $|h|_Y\le |h|_X$, hence
$$
t_{G,Y}(h)\le t_{H,X}(h).
$$
Therefore $\mathcal A\subset\mathcal B$.
Because $G$ is strongly translation discrete,
there are finitely many conjugacy classes in $\mathcal B$.
Hence there are finitely many conjugacy classes in $\mathcal A$
by Proposition~\ref{prop:conj}.
\end{proof}

\subsection{Classification of periodic elements}
\label{sec:PeriodElt}

An element $g$ in a braid group is said to be \emph{periodic}
if it has a central power.
In this subsection, we classify periodic elements in $B(de,e,r)$.
Here the group $B(de,e,r)$ is regarded as a subgroup of the braid group $B_{r+1}$.

\subsubsection{Periodic elements in $B_{r+1}$ and $B_{r+1,1}$}

The center of the Artin braid group $B_{r+1}$ is
an infinite cyclic group generated by $\Delta^2$,
where $\Delta=\sigma_1(\sigma_2\sigma_1)\ldots(\sigma_r\cdots\sigma_1)$.
It is a classical theorem of Brouwer, Ker\'ekj\'art\'o and
Eilenberg~\cite{Bro19, Ker19, Eil34} that an $(r+1)$-braid is periodic
if and only if it is conjugate
to a power of either $\delta$ or $\varepsilon$,
where  $\delta=\sigma_r\cdots\sigma_1$ and $\varepsilon=\delta\sigma_1$.
See Figure~\ref{fig:epsilon}(a,b).

The center of $B_{r+1,1}\cong B({\mathbf{B}}_r)\cong B(d,1,r)$ is
also the infinite cyclic group generated by $\Delta^2$,
and every periodic element of $B({\mathbf{B}}_r)$ is conjugate to a power of $\varepsilon$.

\begin{figure}
$$\begin{array}{ccccc}
\includegraphics[scale=.8]{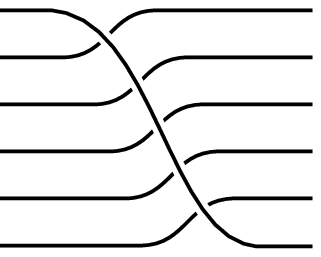} &\qquad&
\includegraphics[scale=.8]{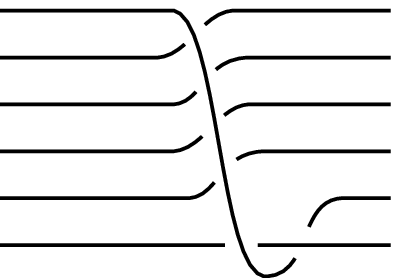} &\qquad&
\includegraphics[scale=.8]{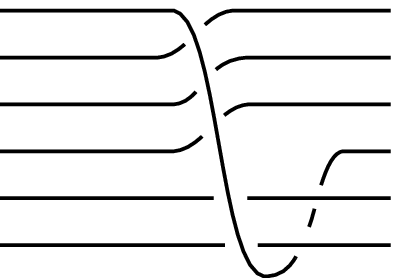}\\
\mbox{(a) The braid $\delta$} &&
\mbox{(b) The braid $\varepsilon$} &&
\mbox{(c) The braid $\varepsilon_1$}
\end{array}
$$
\caption{Braid pictures for $\delta$, $\varepsilon$ and $\varepsilon_1$ in $B_6$}
\label{fig:epsilon}
\end{figure}

Similar statements hold for the braid groups of other finite types
and the braid group $B(e,e,r)$.
For example, see~\cite{LL11}.
Bessis~\cite{Bes06b} explored many
important properties of periodic elements
in the context of braid groups of complex reflection
groups.

\subsubsection{Periodic elements in $B(de,e,r)$}
\label{ssec:per}

The center of $B(de,e,r)$ is an infinite cyclic group generated by
$$
\Delta_{(de,e,r)}=z^{\frac{r}{e\wedge r}} (At_1t_0)^{\frac{e(r-1)}{e\wedge r}},
$$
where $A=s_r\cdots s_3\in B(de,e,r)$~\cite{BMR98}.

\begin{lemma}\label{lem:per-eqns}
Let $A=s_r\cdots s_3\in B(de,e,r)$.
Then the following hold.
\begin{enumerate}
\item[(i)] $At_kAt_{k-1}\cdots At_1=\sigma_1^{-2k}\varepsilon^k$ for all $k\in{\mathbb Z}$.
\item[(ii)] $(At_1t_0)^{r-1}=\sigma_1^{-2r}\Delta^2$.
\item[(iii)] $(At_1t_0)^{r-1}=At_{j+r}\cdots At_{j+1}$ for all $j\in{\mathbb Z}$.
\end{enumerate}
\end{lemma}

\begin{proof}
(i)\ \
Recall that $t_i=\sigma_1^{-2i}\sigma_2\sigma_1^{2i}$ for all $i\in{\mathbb Z}$.
Since for every $i\in{\mathbb Z}$
$$
At_i=A\sigma_1^{-2i}\sigma_2\sigma_1^{2i}
=\sigma_1^{-2i} A\sigma_2\sigma_1^2\sigma_1^{2(i-1)}
=\sigma_1^{-2i} \varepsilon\sigma_1^{2(i-1)},
$$
we have for every $k\in{\mathbb Z}$
$$
At_k At_{k-1}\cdots At_1
= (\sigma_1^{-2k} \varepsilon\sigma_1^{2(k-1)})
(\sigma_1^{-2(k-1)} \varepsilon\sigma_1^{2(k-2)}) \cdots
(\sigma_1^{-2} \varepsilon)
=\sigma_1^{-2k}\varepsilon^k.
$$

(ii)\ \
Let $\varepsilon_1=(\sigma_r\cdots\sigma_2\sigma_1)\sigma_1\sigma_2=A\sigma_2\sigma_1^2\sigma_2$.
See Figure~\ref{fig:epsilon}(c).
Geometrically, $\varepsilon_1$ is the $1/(r-1)$-twist around the first two strings.
Hence $\varepsilon_1^{r-1}$ is the full twist except that
the first two strings are straight, that is,
$$
\varepsilon_1^{r-1}=\sigma_1^{-2}\Delta^2.
$$
Also notice that $\sigma_1$ commutes with $\varepsilon_1$ and $A$. Since
$$
At_1t_0=A (\sigma_1^{-2}\sigma_2\sigma_1^2)\sigma_2
=\sigma_1^{-2}A\sigma_2\sigma_1^2\sigma_2
=\sigma_1^{-2}\varepsilon_1,
$$
we have
$$
(At_1t_0)^{r-1}
=(\sigma_1^{-2}\varepsilon_1)^{r-1}
= \sigma_1^{-2(r-1)}\varepsilon_1^{r-1}
= \sigma_1^{-2(r-1)}\sigma_1^{-2}\Delta^2
= \sigma_1^{-2r}\Delta^2.
$$

(iii)\ \
Setting $k=r$ to (i), we have
\begin{equation}\label{eq:A_t1t0}
At_r\cdots At_1=\sigma_1^{-2r}\varepsilon^r
=\sigma_1^{-2r}\Delta^2.
\end{equation}

Notice that $\tau(A)=A$, $\tau(t_k)=t_{k+1}$ and $\tau(t_1t_0)=t_2t_1=t_1t_0$.
Applying $\tau^j$ to  both sides of \eqref{eq:A_t1t0} and using (ii), we obtain (iii).
\end{proof}

Recall $At_1t_0=\sigma_1^{-2}\varepsilon_1$ from the proof of Lemma~\ref{lem:per-eqns}.
Then the generator $\Delta_{(de,e,r)}$ of the center of $B(de,e,r)$ is
written as
\begin{align*}
\Delta_{(de,e,r)}
&= z^{\frac{r}{e\wedge r}} (At_1t_0)^{\frac{e(r-1)}{e\wedge r}}
= (\sigma_1^{2e})^{\frac{r}{e\wedge r}} (\sigma_1^{-2}\varepsilon_1)^{\frac{e(r-1)}{e\wedge r}}\\
&= (\sigma_1^2)^{\frac{e}{e\wedge r}} (\varepsilon_1^{r-1})^{\frac{e}{e\wedge r}}
= (\sigma_1^2)^{\frac{e}{e\wedge r}} (\sigma_1^{-2}\Delta^2)^{\frac{e}{e\wedge r}}
= (\Delta^2)^{\frac{e}{e\wedge r}}.
\end{align*}
Therefore $\Delta_{(de,e,r)}$ is a power of $\Delta^2$.
Note that $\operatorname{wd}(\Delta_{(de,e,r)})=\frac{er}{e\wedge r}=e\vee r$.
In other words, the center of $B(de,e,r)$ is the intersection of $B(de,e,r)$
and the center of $B_{r+1,1}$.
Hence we have the following.

\begin{proposition}\label{prop:per-Bdeer-Br}
An element $g\in B(de,e,r)$ is periodic in $B(de,e,r)$ if and only if
it is periodic in $B_{r+1,1}$.
\end{proposition}

Now we classify periodic elements in $B(de,e,r)$.

\begin{definition}
Set $\lambda=zAt_eAt_{e-1}\cdots At_1\in B(de,e,r)$.
\end{definition}

Recall that $z=\sigma_1^{2e}$. By Lemma~\ref{lem:per-eqns},
$\lambda=\sigma_1^{2e}(\sigma_1^{-2e}\varepsilon^e)=\varepsilon^e$,
hence $\lambda$ is periodic in $B_{r+1,1}$.
By Proposition~\ref{prop:per-Bdeer-Br}, $\lambda$ is periodic in $B(de,e,r)$.
It follows also from
$$
\lambda^{\frac{r}{e\wedge r}}
=\varepsilon^{\frac{er}{e\wedge r}}
=(\Delta^2)^{\frac{e}{e\wedge r}}
=\Delta_{(de,e,r)}.
$$

\begin{theorem}\label{thm:periodic}
In $B(de,e,r)$, an element $g$ is periodic if and only if
$g$ is conjugate to a power of $\lambda$.
\end{theorem}

\begin{proof}
First, notice that $\lambda$ is conjugate to $\tau(\lambda)$ in $B(de,e,r)$ because
$$
(At_1)\lambda(At_1)^{-1}
= At_1 z At_e\cdots At_2
= z At_{e+1}At_e\cdots At_2
= \tau(\lambda).
$$
Therefore $\lambda^q$ is conjugate to $\tau^k(\lambda^q)$ in $B(de,e,r)$
for all $k,q\in{\mathbb Z}$.

Suppose that $g$ is a periodic element in $B(de,e,r)$.
Then $g$ is periodic in $B_{r+1,1}$ by Proposition~\ref{prop:per-Bdeer-Br},
hence it is conjugate to $\varepsilon^p$ for some $p\in{\mathbb Z}$.
Since $\operatorname{wd}(g)=p\operatorname{wd}(\varepsilon)=p\equiv 0 \mod{e}$, $p=qe$ for some $q\in{\mathbb Z}$.
Then $g$ is conjugate to $\varepsilon^{qe}=\lambda^q$ in $B_{r+1,1}$.
Hence $g$ is conjugate to $\tau^k(\lambda^q)$ in $B(de, e, r)$ for some $0\le k<e$
by Proposition~\ref{prop:conj}.
Therefore $g$ is conjugate to $\lambda^q$ in $B(de,e,r)$.

The converse direction is obvious.
\end{proof}

\subsubsection{Comparison with the results of Bessis}

Here we assume $d,e,r\ge 2$.
We recall the results of Bessis~\cite{Bes06b} on periodic elements
in the braid groups associated with \emph{well-generated} complex reflection groups.
The reflection group $G(de,e,r)$ is \emph{not} well-generated.
However we will see that Bessis' results hold for $B(de,e,r)$
except the existence of the dual Garside element $\delta$.

Let $G$ be a complex reflection group on $V$.
Let $d_1\le d_2\le \cdots \le d_r$ be degrees of $G$
and $d_1^*\ge d_2^*\ge\cdots\ge d_r^*$ be codegrees of $G$.
The largest degree $d_r$ is called the \emph{Coxeter number} of $G$,
which we denote by $h$.
An integer $p$ is called a \emph{regular number}
if there exist an element $g\in G$ and a complex $p$-th root $\zeta$
of unity such that $\ker(g-\zeta)\cap V^{\mathrm{reg}}\ne\emptyset$,
where $V^{\mathrm{reg}}$ is the complement in $V$ of the reflecting hyperplanes of $G$.
The complex reflection group $G$ is called \emph{well-generated} if $G$ can be generated by
$r$ reflections.
It is known that $G$ is well-generated if and only if $G$ is a \emph{duality group},
i.e., $d_i+d_i^*=d_r$ for all $1\le i\le r$.

The following theorem collects Bessis' results on periodic elements in braid groups;
see Lemma 6.11 and Theorems 1.9, 8.2, 12.3, 12.5 in \cite{Bes06b}.
The equivalence between (a) and (b) in the theorem was proved by
Lehrer and Springer~\cite{LS99} and by Lehrer and Michel~\cite{LM03}, and it does not require
well-generatedness of the complex reflection group $G$.

\begin{theorem}[\cite{Bes06b}]\label{thm:bes06a}
Let\/ $G$ be an irreducible well-generated complex reflection group,
with degrees $d_1,\ldots,d_r$, codegrees $d_1^*,\ldots,d_r^*$ and Coxeter number $h$.
Then its braid group $B(G)$ admits the dual Garside structure with
Garside element $\delta$, and the following hold.

\begin{enumerate}
\item The element $\mu=\delta^h$ is central in $B(G)$ and lies in the pure braid group
of $G$.
\item Let $h'=h/(d_1\wedge\cdots\wedge d_r)$.
The center of $B(G)$ is a cyclic group generated by $\delta^{h'}$.
\item Let $p$ be a positive integer, and let
$$
A(p)=\{\, 1\le i\le r \mathrel{:} p\mid d_i\,\}
\quad\mbox{and}\quad
B(p)=\{\, 1\le i\le r \mathrel{:} p\mid d_i^* \,\}.
$$
Then $|A(p)|\le|B(p)|$, and the following conditions are equivalent:
\begin{enumerate}
\item $|A(p)|=|B(p)|$;
\item $p$ is regular;
\item there exists a $p$-th root of $\mu$.
\end{enumerate}
Moreover, the $p$-th root of $\mu$, if exists, is unique up to conjugacy in $B(G)$.
\end{enumerate}
\end{theorem}

The Coxeter groups of types ${\mathbf{A}}_r, {\mathbf{B}}_r, {\mathbf D}_r, \mathbf{I}_2(e)$ and
the complex reflection group $G(e,e,r)$
are all irreducible well-generated complex reflection groups.
But $G(de,e,r)$ is \emph{not} well-generated,
hence the above theorem cannot be applied.

We remark that Bessis' results hold for $B(de,e,r)$ except the existence of $\delta$.
The degrees and codegrees of $G(de,e,r)$ are as follows~\cite{BMR98}:
\begin{align*}
\{d_1,\ldots, d_r\}   &= \{e, 2e,\ldots,(r-1)e, r\} = \{e,2e,\ldots,(r-1)e\}\cup\{r\}, \\
\{d_1^*,\ldots, d_r^*\} &= \{0, e, 2e,\ldots,(r-1)e\} = \{e,2e,\ldots,(r-1)e\}\cup\{0\}.
\end{align*}

Therefore
$d_1\wedge\cdots\wedge d_r=e\wedge r$, $h=e(r-1)$
and $h'=h/(d_1\wedge\cdots\wedge d_r)= e(r-1)/(e\wedge r)$.

\begin{enumerate}
\item Let $\mu=z^r(At_1t_0)^{e(r-1)}$. Then $\mu$ is central in $B(de,e,r)$.
\item The center of $B(de,e,r)$ is an infinite cyclic group
generated by $\Delta_{(de,e,r)}=z^{\frac{r}{e\wedge r}} (At_1t_0)^{\frac{e(r-1)}{e\wedge r}}$.
Notice that $(\Delta_{(de,e,r)})^{d_1\wedge\cdots\wedge d_r}
=(\Delta_{(de,e,r)})^{e\wedge r}=\mu$.

\item The following are equivalent.
\begin{enumerate}
\item $|A(p)|=|B(p)|$;
\item $p$ is regular;
\item there exists a $p$-th root of $\mu$;
\item $p\mid r$.
\end{enumerate}
Moreover, the $p$-th root of $\mu$, if exists, is unique up to conjugacy in $B(de,e,r)$.
\end{enumerate}
In the above, the equivalence between
(c) and (d) follows from Theorem~\ref{thm:periodic}
because $\lambda$ is an $r$-th root of $\mu$.
The other statements are immediate.

\bigskip\noindent
\textbf{Question.}\ \
Can we generalize the approach of Bessis in~\cite{Bes06b} to the braid group $B(de,e,r)$?

\section{$\widetilde{{\mathbf{A}}}$-type presentation for $B(de,e,r)$}
\label{sec:Atilde}

The presentation of $B(de,e,r)$ in Theorem~\ref{new-pre-deer}
describes the semidirect product decomposition
$B(de,e,r)\cong C^e_\infty\ltimes B(\infty,\infty,r)$:
the last two relations describe the action of $z$ on $B(\infty,\infty,r)$,
where $z$ is the generator of the infinite cyclic group $C^e_\infty$,
and the others are the relations of $B(\infty,\infty,r)$.

Because $B(\infty,\infty,r)\cong B({\widetilde{\mathbf A}_{r-1}}
)$,
the group $B(de,e,r)$ has an $\widetilde{{\mathbf{A}}}$-type presentation with
generators $\{s_1, s_2, \ldots, s_r \}$ of $B({\widetilde{\mathbf A}_{r-1}}
)$ along with $z$.
In this section, we explicitly compute this presentation.

Throughout this section, we assume $r\ge 3$
and regard the groups $B(\infty,\infty,r)$, $B({\widetilde{\mathbf A}_{r-1}}
)$ and $B(de,e,r)$
as subgroups of the braid group $B_{r+1}$,
hence $B({\widetilde{\mathbf A}_{r-1}}
)=B(\infty,\infty,r)\subset B(de,e,r)\subset B_{r+1}$.
Let $A$ and $B$ denote the braids
$$A=\sigma_r\sigma_{r-1}\cdots \sigma_3,\qquad
B=A\sigma_2=\sigma_r\sigma_{r-1}\cdots \sigma_3\sigma_2.$$

Recall from \S\ref{ssec:Embed_Bdeer_Br} that the generators
$z\in B(de,e,r)$ and $t_i, s_j\in B(\infty,\infty,r)$ are
\begin{align*}
z  &=\sigma_1^{2e},\\
t_i&=\sigma_1^{-2i}\sigma_2\sigma_1^{2i} \quad\text{for } i\in{\mathbb Z},\\
s_j&=\sigma_j \quad\text{for } 3\le j\le r.
\end{align*}
Recall from \S\ref{ssec:GeomBr} that
the generators $s_1,\ldots,s_r$ of $B(\widetilde{\mathbf{A}}_{r-1})$ are
\begin{align*}
s_1&=At_1A^{-1}=A\sigma_1^{-2}\sigma_2\sigma_1^2A^{-1},\\
s_j&=\sigma_j \quad\text{for } 2\le j\le r.
\end{align*}
Note that $A,B\in B({\widetilde{\mathbf A}_{r-1}}
)$.
See Figure~\ref{fig:br-pic} for the braid pictures of $s_1$, $t_1$ and $t_2$.

\begin{figure}
\begin{tabular}{cccccc}
\includegraphics[scale=.7]{braid-s1.eps}&\qquad&
\includegraphics[scale=.7]{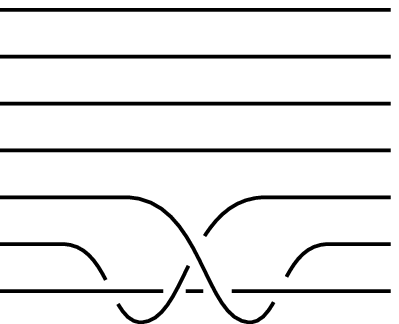}&\qquad&
\includegraphics[scale=.7]{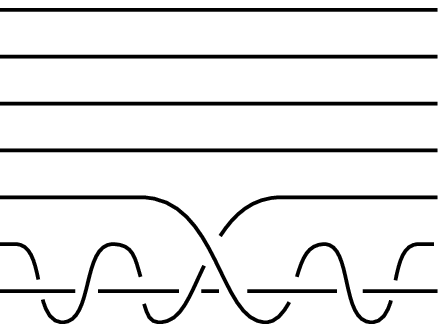}\\
(a) $s_1=At_1A^{-1}$&&
(b) $t_1=\sigma_1^{-2}\sigma_2\sigma_1^2$ &&
(c) $t_2=\sigma_1^{-4}\sigma_2\sigma_1^4$
\end{tabular}
\caption{Braid pictures for $s_1$, $t_1$ and $t_2$ when $r=6$}
\label{fig:br-pic}
\end{figure}

Recall the diagrams for the usual (Artin) presentation of $B({\widetilde{\mathbf A}_{r-1}}
)$
and the new presentation of $B(\infty,\infty,r)$ shown
in Figures~\ref{fig:Atilde_Coxeter} and \ref{fig:Binfty}(b).
These diagrams give natural diagram automorphisms $\kappa$ and $\tau$,
obtained respectively by rotating or by shifting by one node the diagram.
The automorphism $\kappa:B({\widetilde{\mathbf A}_{r-1}}
)\to B({\widetilde{\mathbf A}_{r-1}}
)$ is given by
$s_j\mapsto s_{j+1\bmod r}$ for $j\in{\mathbb Z}/r$.
It is not hard to see from braid relations that $\kappa$ is the conjugation by
$\varepsilon=\sigma_r\cdots\sigma_2\sigma_1^2=A\sigma_2\sigma_1^2=B\sigma_1^2$, that is,
$$\kappa(g)=\varepsilon^{-1}g\varepsilon
\qquad\text{for } g\in B({\widetilde{\mathbf A}_{r-1}}
).
$$
The automorphism $\tau:B(\infty,\infty,r)\to B(\infty,\infty,r)$ is given by
$t_i\mapsto t_{i+1}$ and $s_j\mapsto s_j$ for $i\in{\mathbb Z}$ and $3\le j\le r$.
It is the conjugation by $\sigma_1^2$,
that is,
$$
\tau(g)=\sigma_1^{-2}g\sigma_1^2
\qquad\text{for } g\in B(\infty,\infty,r).
$$
Since $\varepsilon=B\sigma_1^2$, $\kappa$ and $\tau$ are related by
$$\tau(g)=\kappa(BgB^{-1})
\qquad\text{for } g\in B(\infty,\infty,r)=B({\widetilde{\mathbf A}_{r-1}}
).
$$

\begin{lemma}\label{lem:com}
The braids $s_1B$ and $s_j$ for $3\le j\le r$ commute with $\sigma_1$,
hence $\tau(s_1B)=s_1B$ and $\tau(s_j)=s_j$ for $3\le j\le r$.
\end{lemma}

\begin{proof}
Note that $s_1B=(A\sigma_1^{-2}\sigma_2\sigma_1^2A^{-1}) (A\sigma_2)
= A\sigma_1^{-2}\sigma_2\sigma_1^2 \sigma_2
= \sigma_1^{-2} A\sigma_2\sigma_1^2 \sigma_2
= \sigma_1^{-2}\varepsilon_1$,
where the braid $\varepsilon_1$ is illustrated in Figure~\ref{fig:epsilon}(c).
Since $\varepsilon_1$ commutes with $\sigma_1$, so does $s_1B$.
It is obvious that $\sigma_1$ commutes with $s_j=\sigma_j$ for $3\le j\le r$.
\end{proof}

\begin{lemma}\label{lem:s2}
For any $k\ge 1$,
\begin{align*}
\tau^k(s_2)
&=\kappa(B)\kappa^2(B)\cdots\kappa^k(B)
\kappa^k(A^{-1})\kappa^{k-1}(B^{-1})\cdots\kappa(B^{-1})\\
&=\langle s_1B\rangle^{k(r-1)} \left( \langle s_1B\rangle^{k(r-1)-1}\right)^{-1}.
\end{align*}
\end{lemma}

\begin{proof}
We prove the first equality by induction on $k$.
The case $k=1$ is true since
$$\tau(s_2)
=\kappa(Bs_2B^{-1})
=\kappa(Bs_2(As_2)^{-1})
=\kappa(BA^{-1})=
\kappa(B)\kappa(A^{-1}).
$$
Assume that the assertion is true for some $k\ge 1$. Then
\begin{align*}
\tau^{k+1}(s_2)
&=\tau (\tau^k(s_2)) = \kappa(B\tau^k(s_2) B^{-1})\\
&= \kappa(B \cdot \kappa(B)\cdots\kappa^k(B)
\kappa^k(A^{-1})\kappa^{k-1}(B^{-1})\cdots\kappa(B^{-1}) \cdot B^{-1})\\
&= \kappa(B)\kappa^2(B)\cdots \kappa^{k+1}(B)
    \kappa^{k+1}(A^{-1})\kappa^k(B^{-1})\cdots \kappa^2(B^{-1}) \kappa(B^{-1}).
\end{align*}
Therefore the assertion is true for $k+1$.

By a straightforward computation, one can easily see that
$\kappa(B)\kappa^2(B)\cdots\kappa^k(B)=\langle s_1B\rangle^{k(r-1)}$
and that $\kappa(B)\kappa^2(B)\cdots\kappa^{k-1}(B)\kappa^k(A)=\langle s_1B\rangle^{k(r-1)-1}$.
For example,
\begin{align*}
\kappa(B)\kappa^2(B)
&=\kappa(s_r\cdots s_2)\kappa^2(s_r\cdots s_2)\\
&=(s_1s_rs_{r-1}\cdots s_3)(s_2s_1s_r\cdots s_4)=\langle s_1s_r\cdots s_2\rangle^{2(r-1)}\\
&=\langle s_1B\rangle^{2(r-1)}.
\end{align*}
Therefore the second equality is immediate.
\end{proof}

The action of $z$ on $B({\widetilde{\mathbf A}_{r-1}}
)$ is given by
$z^{-1}gz=\tau^e(g)$ for $g\in B({\widetilde{\mathbf A}_{r-1}}
)$.
Notice that $\{ s_1B, s_2, \ldots, s_r \}$ generates $B(\widetilde{\mathbf{A}}_{r-1})$.
Using Lemmas~\ref{lem:com} and \ref{lem:s2}, the action of $z$ on $B({\widetilde{\mathbf A}_{r-1}}
)$  can be described as
\begin{align*}
z^{-1}s_iz  &=\tau^e(s_i)=s_i \quad\text{for}\ 3\le i\le r,\\
z^{-1}s_1Bz &=\tau^e(s_1B)=s_1B,\\
z^{-1}s_2z  &=\tau^e(s_2)=\langle s_1B\rangle^{e(r-1)} \left( \langle s_1B\rangle^{e(r-1)-1}\right)^{-1}.
\end{align*}
These are equivalent to
\begin{align*}
zs_i & = s_iz\quad\text{for}\ 3\le i\le r,\\
zs_1B &= s_1Bz,\\
z\langle s_1B\rangle^{e(r-1)} &= s_2z\langle s_1B\rangle^{e(r-1)-1}.
\end{align*}

This allows us to obtain a positive homogeneous presentation for $B(de,e,r)$
in terms of the conventional generators of ${\widetilde{\mathbf A}_{r-1}}
$.
(We note, however, that the following presentation does not give
rise to a quasi-Garside structure.)

\begin{theorem}
\label{atilde-type-presentation}
The group $B(de,e,r)  \cong C_{\infty}^e \ltimes B(\widetilde{\mathbf{A}}_{r-1})$
for $d\ge 2$, $e\ge 1$ and $r\ge 3$ has the following  presentation:

\begin{itemize}
\item Generators: $\{z\} \cup \{s_i \mid 1\le i\le r\}$;
\item Relations:\\
$\begin{array}{ll}
(A_1) & s_is_j=s_js_i\quad\text{for}\ i-j\not\equiv \pm 1 \mod{r},\\
(A_2) & s_is_{i+1}s_i=s_{i+1}s_is_{i+1} \quad\text{for}\  1\le i\le r, \\
(A_3) & zs_i = s_iz \quad\text{for}\  3\le i\le r\\
(A_4) & zs_1B = s_1Bz,\\
(A_5) & z\langle  s_1B\rangle^{e(r-1)} = s_2z\langle  s_1B\rangle^{e(r-1)-1},
\end{array} $
\end{itemize}
where $B=s_rs_{r-1}\cdots s_2$.
Furthermore, on adding the relations $z^d = 1$ and $s_i^2 = 1$ for all
$1\le i\le r$, a
presentation for the reflection group $G(de,e,r)$ is obtained,
where the generators are all reflections.
\end{theorem}

\section*{Acknowledgments}

The first author gratefully acknowledges the support of a European Union
\emph{Marie Curie} Fellowship while this work was carried out.
The second author was supported by Basic Science Research Program
through the National Research Foundation of Korea (NRF)
funded by the Ministry of Science, ICT \& Future Planning (NRF-2012R1A1A3006304).
The third author was supported by Basic Science Research Program
through the National Research Foundation of Korea (NRF)
funded by the Ministry of Education (NRF-2013R1A1A2007523).

\end{document}